\documentclass[12pt]{article}

\usepackage{graphicx}
\usepackage{latexsym,amsmath,amsfonts,amscd, amsthm}
\usepackage{changebar}
\usepackage{color}
\usepackage{bm}
\usepackage{tikz}
\usepackage{multirow}
\usepackage{epsfig,subfigure}
\usepackage{multirow}
\usepackage{enumerate}
\usetikzlibrary{arrows,backgrounds,snakes}

\allowdisplaybreaks

\newtheoremstyle{plainNoItalics}{}{}{\normalfont}{}{\bfseries}{.}{ }{}

\theoremstyle{plain}
\newtheorem{thm}{Theorem}[section]
\theoremstyle{plainNoItalics}

\newtheorem{lem}[thm]{Lemma}
\newtheorem{defn}[thm]{Definition}
\newtheorem{rem}[thm]{Remark}
\newtheorem{prop}[thm]{Proposition}
\newtheorem{exa}[thm]{Example}

\newcommand{\f}{\frac}

\newcommand{\beq}{\begin{equation}}
\newcommand{\eeq}{\end{equation}}
\newcommand{\beqa}{\begin{eqnarray}}
\newcommand{\eeqa}{\end{eqnarray}}
\newcommand{\bit}{\begin{itemize}}
\newcommand{\eit}{\end{itemize}}
\newcommand{\bedef}{\begin{defn}}
\newcommand{\edefn}{\end{defn}}
\newcommand{\bpro}{\begin{prop}}
\newcommand{\epro}{\end{prop}}

\begin{document}

\baselineskip=1.6pc


\begin{center}
{\bf
High order maximum principle preserving finite volume method 
for convection dominated problems
}
\end{center}
\vspace{.2in}
\centerline{
Pei Yang \footnote{Department of
Mathematics, University of Houston, Houston, 77204. E-mail:
peiyang@math.uh.edu.}
Tao Xiong \footnote{Department of
Mathematics, University of Houston, Houston, 77204. E-mail:
txiong@math.uh.edu.}
Jing-Mei Qiu \footnote{Department of Mathematics, University of Houston,
Houston, 77204. E-mail: jingqiu@math.uh.edu.
The first, second and the third authors are supported by Air Force Office of Scientific Computing YIP grant FA9550-12-0318, NSF grant DMS-1217008.}
Zhengfu Xu \footnote{Department of
Mathematical Science, Michigan Technological University, Houghton, 49931. E-mail: zhengfux@mtu.edu. Supported by NSF grant DMS-1316662.}
}


\bigskip
\centerline{\bf Abstract}

In this paper, we investigate the application of the maximum principle preserving (MPP) parametrized flux limiters to the high order finite volume scheme with Runge-Kutta time discretization for solving convection dominated problems. Such flux limiter was originally proposed in {\em [Xu, Math. Comp., 2013]} and further developed in {\em[Xiong et. al., J. Comp. Phys., 2013]} for finite difference WENO schemes with Runge-Kutta time discretization for convection equations. The main idea is to limit the temporal integrated high order numerical flux toward a first order MPP monotone flux. In this paper, we generalize such flux limiter to high order finite volume methods solving convection-dominated problems, which is easy to implement and introduces little computational overhead. More importantly, for the first time in the finite volume setting, we provide a general proof that the proposed flux limiter maintains high order accuracy of the original WENO scheme for linear advection problems without any additional time step restriction. For general nonlinear convection-dominated problems, 
we prove that the proposed flux limiter introduces up to $\mathcal{O}({\Delta x^3} + \Delta t^3)$ modification to the high order temporal integrated flux in the original WENO scheme without extra time step constraint. We also numerically investigate the preservation of up to ninth order accuracy of the proposed flux limiter in a general setting. The advantage of the proposed method is demonstrated through various numerical experiments. 


\vfill

\newpage

\newpage

\section{Introduction}
\label{sec1}
\setcounter{equation}{0}
\setcounter{figure}{0}
\setcounter{table}{0}

Recently, there is a growing interest in designing high order maximum principle preserving (MPP) schemes for solving scalar convection-dominated problems \cite{zhang2010maximum, zhang2012maximumcd, mpp_xu, jiang2013parametrized, mpp_xuMD, mpp_xqx}, positivity preserving schemes for compressible Euler and Navier-Stokes equations \cite{hu2013positivity, pp_euler, perthame1996positivity, zhang2010positivity}. The motivation of this family of work arises from the observation that many existing high order conservative methods break down when simulating fluid dynamics in extreme cases such as near-vacuum state. To illustrate the purpose of the family of the MPP methods, we shall consider the solution to the following problem
\begin{equation}
u_t+f(u)_x=a(u)_{xx}, \qquad u(x,0)=u_0(x),
\label{ad1}
\end{equation}
with $a'(u)>0$. The solution to (\ref{ad1}) satisfies the maximum principle, i.e.
\begin{equation}\label{maxmin}
\text{if } u_M=\max_x u_0(x), \text{} u_m=\min_x u_0(x), \text{ then } u(x,t)\in [u_m, u_M].
\end{equation}
Within the high order finite volume (FV) Runge-Kutta (RK) weighted essentially non-oscillatory (WENO) framework, we would like to maintain a discrete form of (\ref{maxmin}):
\begin{equation}\label{maxminD}
\text{if } u_M=\max_x u_0(x), \text{} u_m=\min_x u_0(x), \text{ then } \bar u^n_j \in [u_m, u_M] \text{ for any } n, \text{ } j,
\end{equation}
where $\bar u^n_j$ approximates the cell average of the exact solution with high order accuracy on a given $j$th spatial interval at time $t^n$. 

Efforts for designing MPP high order schemes to solve (\ref{ad1}) can be found in recent work by Zhang et al. \cite{zhang2012maximumcd, yzhang2012maximum}, as a continuous research effort to design high order FV and discontinuous Galerkin (DG) MPP schemes based on a polynomial rescaling limiter on the reconstructed (for FV) or representing (for DG) polynomials \cite{zhang2010maximum}. 
This approach requires the updated {\em cell average} to be written as a convex combination of some local quantities within the range $[u_m, u_M]$. For convection-diffusion problems which do not have a finite speed of propagation, it is difficult to generalize such approach to design MPP schemes that are higher than third order accurate. 
In \cite{jiang2013parametrized}, an alternative approach via a parametrized flux limiter, developed earlier by Xu et al. \cite{mpp_xu,mpp_xqx}, is proposed for the finite difference (FD) RK WENO method in solving convection diffusion equations. The flux limiter is applied to convection and diffusion fluxes together to achieve (\ref{maxminD}) for the approximated point values in the finite difference framework. In this paper, we continue our effort in applying the MPP flux limiters to high order FV RK WENO methods to maintain (\ref{maxminD}) with efficiency. Furthermore, we provide some theoretical analysis on the preservation of high order accuracy for the proposed flux limiter in FV framework. Finally, we remark that our current focus is on convection-dominated diffusion problems for which explicit temporal integration proves to be efficient. For the regime of medium to large diffusion, where implicit temporal integration is needed for simulation efficiency, we refer to earlier work in \cite{fujii1973some, farago2005discrete, farago2006discrete, farago2012discrete} and references therein for the construction of the MPP schemes with finite element framework. The generalization of the current flux limiter is not yet available and is subject to future investigation.  

The MPP methods in \cite{zhang2010maximum, mpp_xu, mpp_xqx} are designed base on the observation that first order monotone schemes in general satisfy MPP property (\ref{maxminD}) with proper Courant-Friedrichs-Lewy (CFL) numbers, while regular high order conservative schemes often fail to maintain (\ref{maxminD}). The MPP flux limiting approach is to seek a linear combination of the first order monotone flux with the high order flux, in the hope of that such combination can achieve both MPP property and high order accuracy under certain conditions, e.g. some mild time step constraint. This line of approach is proven to be successful in \cite{mpp_xqx, jiang2013parametrized} for the FD RK WENO schemes and it is later generalized to the high order semi-Lagrangian WENO method for solving the Vlasov-Poisson system \cite{mpp_vp}. A positivity preserving flux limiting approach is developed in \cite{pp_euler} to ensure positivity of the computed density and pressure for compressible Euler simulations. 
Technically, the generalization of such MPP flux limiters from FD WENO \cite{jiang2013parametrized} to FV WENO method is rather straightforward. Taking into the consideration that FV method offers a more natural framework for mass conservation and flexibility in handling irregular computational domain, we propose to apply the MPP flux limiters to the high order FV RK WENO method to solve (\ref{ad1}). The proposed flux limiting procedure is rather easy to implement even with the complexity of the flux forms in multi-dimensional FV computation. Moreover, a general theoretical proof on preserving both MPP and high order accuracy without additional time step constraint can be done for FV methods when solving a linear advection equation; such result does not hold for high order FD schemes \cite{mpp_xqx}. 


In this paper,  
for the first time, we establish a general proof that, there is no further time step restriction, besides the CFL condition under the linear stability requirement, to preserve high order accuracy when the high order flux is limited toward an upwind first order flux for solving linear advection problem, when the parametrized flux limiters are applied to FV RK WENO method. In other words, both the MPP property and high order accuracy of the original scheme can be maintained without additional time step constraint. For a general nonlinear convection problem, we prove that the flux limiter preserves up to third order accuracy and the discrete maximum principle with no further CFL restriction. {This proof relies on tedious Taylor expansions, and it is difficult to generalize it to results with higher order accuracy (fourth order or higher). On the other hand, such analysis can be extended to a convection-dominated diffusion problem as done in \cite{jiang2013parametrized}.}
Furthermore, numerical results indicate that mild CFL restriction is needed for the MPP flux limiting finite volume scheme without sacrificing accuracy. For more discussions, see Section 3.  

The paper is organized as follows. In Section 2, we provide the numerical algorithm of the high order FV RK WENO schemes with MPP flux limiters. In Section 3, theoretical analysis is given for a linear advection problem and general nonlinear problems. Numerical experiments are demonstrated in Section 4. We give a brief conclusion in Section 5.

\section{A MPP FV method}
\label{sec2}
\setcounter{equation}{0}
\setcounter{figure}{0}
\setcounter{table}{0}

In this section, we propose a high order FV scheme for the convection-diffusion equation. In the proposed scheme, the high order WENO reconstruction of flux is used for the convection term, while a high order compact reconstruction of flux is proposed for the diffusion term. 

For simplicity, we first consider a one dimensional (1D) case. The following uniform spatial discretization is used for a 1D bounded domain $[a, b]$,
\begin{equation}
a=x_{\frac{1}{2}}<x_{\frac{3}{2}}<\cdots<x_{N-\frac{1}{2}}<x_{N+\frac{1}{2}}=b, \ \Delta x =\frac{b-a}{N}.
\end{equation}
with the computational cell and cell center defined as
\begin{equation}
I_j=[x_{j-\frac{1}{2}}, x_{j+\frac{1}{2}}], \ x_j=\frac{1}{2}(x_{j-\frac{1}{2}}+x_{j+\frac{1}{2}}), \ j=1,2,\cdots,N.
\end{equation}
Let $\bar{u}_j$ denote approximation to the cell average of $u$ over cell $I_j$. The FV scheme is designed by integrating equation (\ref{ad1}) over each computational cell $I_j$ and then dividing it by $\Delta x$,
\begin{equation}\label{1dscheme}
\frac{d\bar{u}_j}{dt}  = -\frac{1}{\Delta x} (\hat{H}_{j+\frac{1}{2}}^C - \hat{H}_{j-\frac{1}{2}}^C) + \frac{1}{\Delta x} (\hat{H}_{j+\frac{1}{2}}^D - \hat{H}_{j-\frac{1}{2}}^D),
\end{equation}
where $\hat{H}_{j+\frac{1}{2}}^C$ and $\hat{H}_{j+\frac{1}{2}}^D$ are the numerical fluxes for convection and diffusion terms respectively.

For the convection term, one can adopt any monotone flux. For example, in our simulations, we use the Lax-Friedrichs flux
\begin{equation}
\hat{H}_{j+\frac{1}{2}}^C (u_{j+\frac{1}{2}}^-, u_{j+\frac{1}{2}}^+)= \frac{1}{2}\big(f(u_{j+\frac{1}{2}}^-) + \alpha u_{j+\frac{1}{2}}^-\big) + \frac{1}{2}\big(f(u_{j+\frac{1}{2}}^+) - \alpha u_{j+\frac{1}{2}}^+\big), \ \alpha ={\max}_{u_{m}\le u \le u_{M}} |f'(u)|.
\end{equation}
Here $u_{j+\frac{1}{2}}^- \doteq P(x_{j+\frac{1}{2}})$, where $P(x)$ is obtained by reconstructing a $(2k+1)^{th}$ order polynomial whose averages agree with those in a left-biased stencil $\{\bar{u}_{j-k}, \cdots, \bar{u}_{j+k}\}$,
\[
\frac{1}{\Delta x}\int_{I_{l}} P(x)dx = \bar{u}_{l},\ l=j-k,\cdots,j+k.
\]
The reconstruction procedure for $u_{j+\frac{1}{2}}^+$ can be done similarly from a right-biased stencil. To suppress oscillation around discontinuities and maintain high order accuracy around smooth regions of the solution, the WENO mechanism can be incorporated in the reconstruction. Details of such procedure can be found in $\cite{Shu_book}$. 

For the diffusion term, we propose the following {\em compact} reconstruction strategy for approximating fluxes at cell boundaries $a(u)_x|_{x_{j+\frac12}}$. Without loss of generality, we consider a fourth order reconstruction, while similar strategies can be extended to schemes with arbitrary high order. Below we let $u_j$ denote approximation to the point values of $u$ at $x_j$.
\begin{enumerate}
\item Reconstruct $\{ u_{l}\}_{l=j-1}^{j+2}$ from the cell averages $\{ \bar u_{l}\}_{l=j-1}^{j+2}$ by constructing a cubic polynomial $P(x)$, such that
\[
\frac{1}{\Delta x}\int_{I_{l}} P(x)dx = \bar{u}_{l},\ l=j-1,\cdots,j+2.
\]
Then $u_l = P(x_l)$, $l=j-1, \cdots j+2$. We use $\mathcal{R}_1$ to denote such reconstruction procedure,
    $$( u_{j-1}, u_{j}, u_{j+1}, u_{j+2}) = \mathcal{R}_1(\bar u_{j-1}, \bar u_{j}, \bar u_{j+1}, \bar u_{j+2}).$$
As a reference, the reconstruction formulas for $\mathcal{R}_1$ are provided below,
\begin{align}
& u_{j-1}=\frac{11}{12}\bar{u}_{j-1} + \frac{5}{24}\bar{u}_{j} -\frac{1}{6}\bar{u}_{j+1} + \frac{1}{24}\bar{u}_{j+2}, \nonumber \\
& u_j=-\frac{1}{24}\bar{u}_{j-1} + \frac{13}{12}\bar{u}_{j} -\frac{1}{24}\bar{u}_{j+1}, \nonumber\\
& u_{j+1}=-\frac{1}{24}\bar{u}_{j} + \frac{13}{12}\bar{u}_{j+1} - \frac{1}{24}\bar{u}_{j+2}, \nonumber\\
& u_{j+2}=\frac{1}{24}\bar{u}_{j-1} - \frac{1}{6}\bar{u}_{j} + \frac{5}{24}\bar{u}_{j+1} + \frac{11}{12}\bar{u}_{j+2}. \nonumber
\end{align}
\item Construct an interplant $Q(x)$ such that
    $$Q(x_l)=a(u_l),\ l=j-1,\cdots, j+2.$$
Then let
    $$\hat{H}_{j+\frac{1}{2}}^D=Q'(x)|_{x_{j+\frac{1}{2}}}.$$
Such procedure is denoted as
\[\hat{H}_{j+\frac{1}{2}}^D
=
\mathcal{R}_2(a(u_{j-1}), a(u_{j}), a(u_{j+1}), a(u_{j+2})).
\]
As a reference, we provide the formula for $\mathcal{R}_2$ below
\begin{align}
& \hat{H}_{j+\frac{1}{2}}^D = \frac{1}{24} a(u_{j-1}) - \frac{9}{8} a(u_j) +\frac{9}{8} a(u_{j+1}) - \frac{1}{24} a(u_{j+2}). \nonumber
\end{align}
\end{enumerate}
\begin{rem}
The reconstruction processes for $\mathcal{R}_1$ and $\mathcal{R}_2$ operators are designed such that $\hat{H}_{j+\frac{1}{2}}^D$ is reconstructed from a compact stencil with a given order of accuracy.  Because of such design, for the linear diffusion term $a(u)=u$, $\mathcal{R}_1$ and $\mathcal{R}_2$ can be combined and the strategy above turns out to be a classical fourth order central difference from a five-cell stencil with
 $$\hat{H}_{j+\frac{1}{2}}^D = \frac{1}{\Delta x}(\frac{1}{2} \bar{u}_{j-1} -\frac{15}{12} \bar{u}_j + \frac{15}{12}\bar{u}_{j+1} -\frac{1}{12}\bar{u}_{j+2}).$$
\end{rem}
If  each of $u_l$ ($l=j-1, \cdots j+2$) in Step 1 is reconstructed from symmetrical stencils (having the same number of cells from left and from right), the reconstruction of $\hat{H}_{j+\frac{1}{2}}^D$ will depend on a much wider stencil $\{u_{j-3}, \cdots u_{j+4}\}$.
Such non-compact way of reconstructing numerical fluxes for diffusion terms will introduce some numerical instabilities when approximating nonlinear diffusion terms in our numerical tests, whereas the proposed compact strategy does not encounter such difficulty.
%
%

We use the following third order total variation diminishing (TVD) RK method \cite{gottlieb2009high} for the time discretization of (\ref{1dscheme}), which reads
\begin{align}\label{rk}
&u^{(1)}=\bar u^n+\Delta t L(\bar u^n), \nonumber \\
&u^{(2)}=\bar u^n+\Delta t (\frac{1}{4}L(\bar u^n) + \frac{1}{4}L(u^{(1)})), \\
&\bar u^{n+1}=\bar u^n+\Delta t (\frac{1}{6} L(\bar u^n) + \frac{1}{6} L(u^{(1)}) + \frac{2}{3} L(u^{(2)})), \nonumber
\end{align}
where $L(\bar u^n)$ denotes the right hand side of equation (\ref{1dscheme}). Here $\bar u^n$ and $u^{(s)}$, $s = 1, 2$ denote the numerical solution of $u$ at time $t^n$ and corresponding RK stages.
The fully discretized scheme \eqref{rk} can be rewritten as
\begin{equation}\label{1dschemefully}
\bar u_j^{n+1}=\bar u_j^{n}-\lambda (\hat{H}_{j+\frac{1}{2}}^{rk} - \hat{H}_{j-\frac{1}{2}}^{rk})
\end{equation}
with $\lambda = \frac{\Delta t}{\Delta x}$ and
\[
\hat{H}_{j+\frac{1}{2}}^{rk}=\frac{1}{6}(\hat{H}_{j+\frac{1}{2}}^{C,n} - \hat{H}_{j+\frac{1}{2}}^{D,n}) + \frac{1}{6}(\hat{H}_{j+\frac{1}{2}}^{C,(1)} - \hat{H}_{j+\frac{1}{2}}^{D,(1)}) + \frac{2}{3}(\hat{H}_{j+\frac{1}{2}}^{C,(2)} - \hat{H}_{j+\frac{1}{2}}^{D,(2)}).
\]
Here $\hat{H}_{j+\frac{1}{2}}^{C,(s)},\ \hat{H}_{j+\frac{1}{2}}^{D,(s)}\ (s=1,2)$ are the numerical fluxes at the intermediate stages in the RK scheme (\ref{rk}).

It has been known that the numerical solutions from schemes with a first order monotone flux for the convection term together with a first order flux for the diffusion term satisfy the maximum principle, if the time step is small enough \cite{yzhang2012maximum}. However, if the numerical fluxes are of high order such as the one from the reconstruction process proposed above, the MPP property for the numerical solutions does not necessarily hold under the same time step constraint. Next we apply the parametrized flux limiters proposed in \cite{mpp_xqx} to the scheme (\ref{1dschemefully}) to preserve the discrete maximum principle (\ref{maxminD}).

We modify the numerical flux $\hat{H}_{j+\frac{1}{2}}^{rk}$ in equation (\ref{1dschemefully}) with
\begin{equation}\label{modified flux}
\tilde{H}_{j+\frac{1}{2}}^{rk}=\theta_{j+\frac{1}{2}} \hat{H}_{j+\frac{1}{2}}^{rk} + (1-\theta_{j+\frac{1}{2}}) \hat{h}_{j+\frac{1}{2}},
\end{equation}
by carefully seeking local parameters $\theta_{j+\frac{1}{2}}$, such that the numerical solutions enjoy the MPP property yet $\theta_{j+\frac12}$ is as close to $1$ as possible. In other words, $\tilde{H}_{j+\frac{1}{2}}^{rk}$ is as close to the original high order flux $\hat{H}_{j+\frac{1}{2}}^{rk}$ as possible. 
Here $\hat{h}_{j+\frac{1}{2}}$ denotes the first order flux for convection and diffusion terms, using which in the scheme \eqref{1dscheme} with a forward Euler time discretization guarantees the maximum principle of numerical solutions. For example, we can take
\[
\hat{h}_{j+\frac{1}{2}}=\hat{h}^C_{j+\frac{1}{2}} - \hat{h}^D_{j+\frac{1}{2}} =
\frac{1}{2}\big(f(\bar{u}_{j}) + \alpha \bar{u}_{j}\big) + \frac{1}{2}\big(f(\bar{u}_{j+1}) - \alpha \bar{u}_{j+1}\big)- \frac{a(\bar{u}_{j+1})-a(\bar{u}_j)}{\Delta x}
\]
with $\alpha = {\max}_{u_{m}\le u \le u_{M}} |f'(u)|$. The goal of the procedures outlined below is to adjust $\theta_{j+\frac12}$, so that with the modified flux $\tilde{H}_{j+\frac{1}{2}}^{rk}$, the numerical solutions satisfy the maximum principle, 
\begin{equation}
u_m \le \bar u_j^n - \lambda (\tilde{H}_{j+\frac{1}{2}}^{rk} - \tilde{H}_{j-\frac{1}{2}}^{rk}) \le u_M, \quad \forall j.
\end{equation}
Detailed procedures in decoupling the above inequalities have been intensively discussed in our previous work, e.g. \cite{mpp_xqx}. Below we only briefly describe the computational algorithm for the proposed limiter. 

Let $F_{j+\frac{1}{2}} \doteq  \hat{H}_{j+\frac{1}{2}}^{rk}-\hat{h}_{j+\frac{1}{2}}$ and
\[
\Gamma _j^M \doteq u_M - (\bar u_j^n - \lambda (\hat{h}_{j+\frac{1}{2}} - \hat{h}_{j-\frac{1}{2}})), \ \Gamma _j^m \doteq u_m - (\bar u_j^n - \lambda (\hat{h}_{j+\frac{1}{2}} - \hat{h}_{j-\frac{1}{2}})).
\]
The MPP property is satisfied with the modified flux (\ref{modified flux}) when the following inequalities are hold,
\beqa
&& \lambda \theta_{j-\frac{1}{2}} F_{j-\frac{1}{2}} - \lambda \theta_{j+\frac{1}{2}} F_{j+\frac{1}{2}}  -  \Gamma_j^M \le 0, \label{max} \\
&& \lambda \theta_{j-\frac{1}{2}} F_{j-\frac{1}{2}} - \lambda \theta_{j+\frac{1}{2}}F_{j+\frac{1}{2}}  -  \Gamma_j^m \ge 0.  \label{min}
\eeqa
We first consider the inequality \eqref{max}. We seek a local pair of numbers $(\Lambda _{-\frac{1}{2},I_j} ^M,  \Lambda _{+\frac{1}{2},I_j} ^M)$ such that (1) 
$\Lambda _{\pm\frac{1}{2},I_j} ^M\in[0, 1]$ and is as close to $1$ as possible, (2) for any
$\theta_{j-\frac{1}{2}} \in [0,\Lambda _{-\frac{1}{2},I_j} ^M], \  \theta_{j+\frac{1}{2}} \in [0,\Lambda _{+\frac{1}{2},I_j} ^M]$,
the inequality (\ref{max}) holds. The inequality \eqref{max} can be decoupled based on the following four different cases:
\begin{enumerate}[(a)]
\item  If $F_{j-\frac{1}{2}}\le 0$ and $F_{j+\frac{1}{2}} \ge 0$, then
         $(\Lambda _{-\frac{1}{2},I_j} ^M,  \Lambda _{+\frac{1}{2},I_j} ^M)=(1,1). $
\item  If $F_{j-\frac{1}{2}}\le 0$ and $F_{j+\frac{1}{2}} < 0$, then
         $(\Lambda _{-\frac{1}{2},I_j} ^M,  \Lambda _{+\frac{1}{2},I_j} ^M)=(1, \min(1, \frac{\Gamma_j^M}{-\lambda F_{j+\frac{1}{2}}})). $
\item If $F_{j-\frac{1}{2}}\ > 0$ and $F_{j+\frac{1}{2}} \ge 0$, then
         $(\Lambda _{-\frac{1}{2},I_j} ^M,  \Lambda _{+\frac{1}{2},I_j} ^M)=(\min(1, \frac{\Gamma_j^M}{\lambda F_{j-\frac{1}{2}}}), 1). $
\item If $F_{j-\frac{1}{2}}\ > 0$ and $F_{j+\frac{1}{2}} < 0$, then
         $$(\Lambda _{-\frac{1}{2},I_j} ^M,  \Lambda _{+\frac{1}{2},I_j} ^M)=(\min(1,\frac{\Gamma_j^M}{\lambda F_{j-\frac{1}{2}}-\lambda F_{j+\frac{1}{2}}}), \min(1,\frac{\Gamma_j^M}{\lambda F_{j-\frac{1}{2}}-\lambda F_{j+\frac{1}{2}}})). $$
\end{enumerate}
Similarly, we can find a local pair of numbers $(\Lambda _{-\frac{1}{2},I_j} ^m,  \Lambda _{+\frac{1}{2},I_j} ^m)$ such that for any
$$\theta_{j-\frac{1}{2}} \in [0,\Lambda _{-\frac{1}{2},I_j} ^m], \  \theta_{j+\frac{1}{2}} \in [0,\Lambda _{+\frac{1}{2},I_j} ^m]$$
(\ref{min}) holds. There are also four different cases:
\begin{enumerate}[(a)]
\item If $F_{j-\frac{1}{2}} \ge 0$ and $F_{j+\frac{1}{2}} < 0$, then
          $(\Lambda _{-\frac{1}{2},I_j} ^m,  \Lambda _{+\frac{1}{2},I_j} ^m)=(1,1).$
\item If $F_{j-\frac{1}{2}} \ge 0$ and $F_{j+\frac{1}{2}} > 0$, then
          $(\Lambda _{-\frac{1}{2},I_j} ^m,  \Lambda _{+\frac{1}{2},I_j} ^m)=(1,\min(1,\frac{\Gamma_j^m}{-\lambda F_{j+\frac{1}{2}}})).$
\item If $F_{j-\frac{1}{2}} < 0$ and $F_{j+\frac{1}{2}} < 0$, then
          $(\Lambda _{-\frac{1}{2},I_j} ^m,  \Lambda _{+\frac{1}{2},I_j} ^m)=(\min(1,\frac{\Gamma_j^m}{\lambda F_{j-\frac{1}{2}}}),1).$
\item  If $F_{j-\frac{1}{2}} < 0$ and $F_{j+\frac{1}{2}} \ge 0$, then
          \[(\Lambda _{-\frac{1}{2},I_j} ^m,  \Lambda _{+\frac{1}{2},I_j} ^m)=(\min(1,\frac{\Gamma_j^m}{\lambda F_{j-\frac{1}{2}}-\lambda F_{j+\frac{1}{2}}}), \min(1,\frac{\Gamma_j^m}{\lambda F_{j-\frac{1}{2}}-\lambda F_{j+\frac{1}{2}}})).\]
\end{enumerate}

\noindent Finally, the local limiter parameter $\theta_{j+\frac{1}{2}}$ at the cell boundary $x_{j+\frac{1}{2}}$ is defined as
\begin{equation}
\theta_{j+\frac{1}{2}} = min(\Lambda_{+\frac{1}{2},I_j}^M, \Lambda_{+\frac{1}{2},I_j}^m, \Lambda_{-\frac{1}{2},I_{j+1}}^M, \Lambda_{-\frac{1}{2},I_{j+1}}^m),
\end{equation}
so that the numerical solutions $\bar u_j^{n+1}$, $\forall j, n$  satisfy the maximum principle.


The extension of the FV RK scheme and the MPP flux limiter from 1D case to two dimensional (2D) convection-diffusion problems is straightforward. For example, we consider a 2D problem on a rectangular domain $[a, b] \times [c, d]$,
\beq
\label{eq: 2d}
u_t + f(u)_x + g(u)_y = a(u)_{xx} + b(u)_{yy}.
\eeq
Without loss of generality, we consider a set of uniform mesh
$$a=x_{\frac{1}{2}}<x_{\frac{3}{2}}<\cdots<x_{N-\frac{1}{2}}<x_{N_x+\frac{1}{2}}=b, \ \Delta x =\frac{b-a}{N_x},$$
$$c=y_{\frac{1}{2}}<y_{\frac{3}{2}}<\cdots<y_{N-\frac{1}{2}}<y_{N_y+\frac{1}{2}}=d, \ \Delta y =\frac{d-c}{N_y},$$
with $I_{i,j}=[x_{i-\frac{1}{2}},x_{i+\frac{1}{2}}]\times [y_{j-\frac{1}{2}},y_{j+\frac{1}{2}}]$. A semi-discrete FV discretization of \eqref{eq: 2d} gives
\begin{align}\label{2dscheme}
\frac{d}{dt} \bar{u}_{i,j} & + \frac{1}{\Delta x} (\hat{f}_{i+\frac{1}{2},j}-\hat{f}_{i-\frac{1}{2},j})+ \frac{1}{\Delta y} (\hat{g}_{i,j+\frac{1}{2}}-\hat{g}_{i,j-\frac{1}{2}}) \nonumber \\
&=\frac{1}{\Delta x} (\widehat{(a_x)}_{i+\frac{1}{2},j}-\widehat{(a_x)}_{i-\frac{1}{2},j}) +  \frac{1}{\Delta y} (\widehat{(b_y)}_{i,j+\frac{1}{2}}-\widehat{(b_y)}_{i,j-\frac{1}{2}}),
\end{align}
where $\bar{u}_{i,j}=\frac{1}{\Delta x \Delta y} \int \int_{I_{i,j}} u dxdy$ and $\hat{f}_{i+\frac{1}{2},j}=\frac{1}{\Delta y} \int_{y_{j-\frac{1}{2}}}^{y_{j+\frac{1}{2}}} f(x_{i+\frac{1}{2}},y)dy$ is the average of the flux over the right boundary of cell $I_{i,j}$. $\hat{g}_{i,j+\frac{1}{2}}$, $\widehat{(a_x)}_{i+\frac{1}{2},j}$, $\widehat{(b_y)}_{i,j+\frac{1}{2}}$ can be defined similarly. The flux $\hat{f}_{i+\frac{1}{2},j}$ is evaluated by applying the Gaussian quadrature rule for integration,
\begin{align}
&  \hat{f}_{i+\frac{1}{2},j} = \frac{1}{2} \underset{i_g}{\Sigma} \omega_{i_g} f(u_{i+\frac{1}{2},i_g}).
\end{align}
Here $\underset{i_g}{\Sigma}$ represents the summation over the Gaussian quadratures with $\omega_{i_g}$ being quadrature weights
and $u_{i+\frac{1}{2},i_g}$ is the approximated value to $u(x_{i+\frac{1}{2}},y_{i_g})$ with $y_{i_g}$ being the Gaussian quadrature points over $[y_{j-\frac{1}{2}},y_{j+\frac{1}{2}}]$. $u_{i+\frac{1}{2},i_g}$ can be reconstructed from $\{\bar{u}_{i,j}\}$ in the following two steps. Firstly, we reconstruct $\frac{1}{\Delta x} \int_{x_{i-\frac{1}{2}}}^{x_{i+\frac{1}{2}}} u(x,y_{i_g})dx$ from  $\{\bar{u}_{i,j}\}$. To do this, we construct a polynomial $Q(y)$ such that
\begin{equation}
\frac{1}{\Delta y} \int_{y_{j-\frac{1}{2}}}^{y_{j+\frac{1}{2}}} Q(y)dy =\frac{1}{\Delta x\Delta y} \int_{I_{i,j}} u(x,y)dx dy =\bar{u}_{i,j},
\end{equation}
with $j$ belongs to a reconstruction stencil in the $y$-direction as in the one-dimensional case. 
Then $Q(y_{i_g})$ is a high order approximation to $\frac{1}{\Delta x} \int_{x_{i-\frac{1}{2}}}^{x_{i+\frac{1}{2}}} u(x,y_{i_g})dx$.
We let $\mathcal{R}_{y}$ to denote such reconstruction process in $y$-direction.
Secondly, we construct a polynomial $P(x)$ such that
\begin{equation}
\frac{1}{\Delta x} \int_{x_{i-\frac{1}{2}}}^{x_{i+\frac{1}{2}}} P(x)dx = \frac{1}{\Delta x} \int_{x_{i-\frac{1}{2}}}^{x_{i+\frac{1}{2}}} u(x,y_{i_g})dx,
\end{equation}
with $i$ belongs to a reconstruction stencil in the $x$-direction as in the one-dimensional case. Then $u_{i+\frac{1}{2},i_g}=P(x_{i+\frac{1}{2}})$. 
Such 1D reconstruction process is denoted as $\mathcal{R}_{x}$.
The 2D reconstructing procedure can be summarized as the following flowchart 
\begin{equation}
\centering
\{ \bar{u}_{i,j}\} \overset{\mathcal{R}_y}{\longrightarrow}  \{\frac{1}{\Delta x} \int_{x_{i-\frac{1}{2}}}^{x_{i+\frac{1}{2}}} u(x,y_{i_g})dx \} \overset{\mathcal{R}_x}{\longrightarrow} \{u_{i+\frac{1}{2},i_g}\}.
\end{equation}
Detailed information on the 2D reconstruction procedure is similar to those described in \cite{Shu_book}.  The 2D MPP flux limiter is applied in a similar fashion as those in \cite{mpp_xuMD, jiang2013parametrized, mpp_xqx}. Thus details are omitted for brevity. 

\begin{rem}
The proposed generalization of the parametrized flux limiter to convection-diffusion problems is rather straightforward. In comparison, it is much more difficult to generalize the polynomial rescaling approach in \cite{zhang2010maximum} to schemes with higher than third order accuracy for convection diffusion problems. The approach there relies on rewriting the updated cell average as a convex combination of some local quantities within the range $[u_m, u_M]$; this is more difficult to achieve with the diffusion terms \cite{zhang2012maximumcd, yzhang2012maximum}. Moreover, the proposed flux limiter introduces very mild time step constraint to preserve both MPP and high order accuracy of the original FV RK scheme, see the next section for more discussions. 
\end{rem}

\section{Theoretical properties}
\label{sec3}
\setcounter{equation}{0}
\setcounter{figure}{0}
\setcounter{table}{0}

In this section, we provide accuracy analysis for the MPP flux 
limiter applied to the high order FV RK scheme solving pure convection problems.
Specifically, we will prove that the proposed parametrized flux limiter as in equation \eqref{modified flux} introduces a high order modification in space and time to the temporal integrated flux of the original scheme, assuming that the solution is smooth enough. A general proof on preservation of {\em arbitrary} high order accuracy will be provided for linear problems. Then by performing Taylor expansions around extrema, 
we prove that the modification from the proposed flux limiter is of at least third order, for FV RK schemes that are third order or higher in solving general nonlinear problems.

The entropy solution $u(x, t)$ to a scalar convection problem
\beq
u_t+f(u)_x=0, \quad u(x,0)=u_0(x).
\label{eq: adv}
\eeq 
satisfies 
\begin{eqnarray}
\label{eq: weak}
\frac{d}{dt} \int_{x_{j-\frac12}}^{x_{j+\frac12}} u(x,t) dx =f(u(x_{j+\frac12}, t))-f(u(x_{j-\frac12}, t)).
\end{eqnarray}
Integrating (\ref{eq: weak}) over the time period $[t^n, t^{n+1}]$, we have
\begin{eqnarray}
\label{eq: int}
\bar u_j(t^{n+1}) = \bar u_j(t^n)-\lambda (\check f_{j+\frac12} -\check f_{j-\frac12}),
\end{eqnarray}
where $\lambda=\Delta t/\Delta x$ and
\begin{equation}
\label{cellaverage}
\bar u_j(t)=\frac{1}{\Delta x} \int_{x_{j-1/2}}^{x_{j+1/2}} u(x, t) dx, \quad
\check f_{j-1/2} =\frac{1}{\Delta t} \int_{t^n}^{t^{n+1}} f(u(x_{j-1/2}, t)) dt. 
\end{equation}
The entropy solution satisfies the maximum principle in the form of
\begin{equation}
\label{IQL}
u_m\le\bar u_j(t^n)-\lambda (\check f_{j+\frac12} -\check f_{j-\frac12}) \le u_M.
\end{equation} 

For schemes with $(2k+1)^{th}$ order finite volume 
spatial discretization \eqref{1dschemefully} and $p^{th}$ order RK time discretization, we assume  
\beq 
|\check f_{j+\frac12}-\hat H^{rk}_{j+\frac12}|=\mathcal{O}(\Delta x^{2k+1} + \Delta t^p), \quad \forall j.
\label{assumption}
\eeq
Our analysis is in the sense of local truncation analysis assuming the difference between $\bar u_j(t^n)$ and $\bar u^n_j$ is of high order ($\mathcal{O}(\Delta x^{2k+1} + \Delta t^p)$). Under a corresponding $(2k+1)^{th}$ order reconstruction, the difference between the point values $u(x_j, t^n)$ and $u^n_j$ is also of high order. In the following, we use them interchangeably when such high order difference allows.

For the MPP flux limiter, we only consider the maximum value part as in equation \eqref{max}. 
The proof of equation \eqref{min} for the minimum value would be similar. We would like to prove that the difference 
between $\hat H^{rk}_{j+\f12}$ and $\tilde H^{rk}_{j+\f12}$ in \eqref{modified flux} is of high order in both space and time, that is
\beq
\label{trc}
|\hat H^{rk}_{j+\frac12}-\tilde H^{rk}_{j+\frac12}|=\mathcal{O}(\Delta x^{2k+1} + \Delta t^p), \quad \forall j.
\eeq
There are four cases of the maximum value part \eqref{max} outlined in the previous section. The estimate \eqref{trc} can be easily checked for
case (a) and (d) under the assumption \eqref{assumption} and the fact (\ref{IQL}), see arguments in \cite{mpp_xqx}. Below we will only discuss case (b), as the argument for case (c) would be similar. 

First we give the following lemma: 
\begin{lem}
\label{lem1}
Consider applying the MPP flux limiter \eqref{modified flux} for the maximum value part \eqref{max} with case (b), to prove \eqref{trc}, it suffices to have 
\beq
|u_M-(\bar u_j-\lambda(\check f_{j+\frac{1}{2}}-\hat h_{j-\frac12}))| = \mathcal{O}(\Delta x^{2k+1} + \Delta t^p),
\label{3rd}
\eeq
if $u_M-(\bar u_j-\lambda(\hat H^{rk}_{j+\frac{1}{2}}-\hat h_{j-\frac12})) < 0$.
\end{lem}

\begin{proof}
For case (b), we are considering the case when
\begin{equation*}
\Lambda_{+\frac12, {I_j}}=\frac{\Gamma^M_j}{-\lambda F_{j+\frac12}} < 1,
\end{equation*}
which is equivalent to $u_M-(\bar u_j-\lambda(\hat{H}^{rk}_{j+\frac{1}{2}}-\hat h_{j-\frac12})) < 0$, and
\begin{equation*}
\qquad \tilde H^{rk}_{j+\frac12} -\hat H^{rk}_{j+\frac12} = \frac {\Gamma^M_j+\lambda F_{j+\frac12}}{-\lambda}=\frac{u_M-(\bar u_j -\lambda (\hat H^{rk}_{j+\frac12}-\hat h_{j-\frac12}))}{-\lambda},
\end{equation*}
which indicates that it suffices to have \eqref{3rd} to obtain (\ref{trc}) with the assumption (\ref{assumption}).
\end{proof}

\begin{thm}
\label{general_proof}
Assuming $f'(u) >0$ and $\lambda \max_u |f'(u)|\le 1$, we have 
\beq
\label{UL}
\bar u_j(t^n)-\lambda (\check f_{j+\frac12} - f(\bar u_{j-1}(t^n))) \le u_M
\eeq 
if $u(x, t)$ is the entropy solution to (\ref{eq: adv}) subject to initial data $u_0(x)$.
\end{thm}

\begin{proof}
Consider the problem (\ref{eq: adv}) with a different initial condition at time level $t^n$,
\begin{eqnarray}
\tilde u(x,t^n) = \begin{cases} u(x,t^n) \quad  & x\ge x_{j-\frac12},\\
\bar u_{j-1}(t^n) \quad & x< x_{j-\frac12},
\end{cases}
\end{eqnarray}
here $u(x,t^n)$ is the exact solution of (\ref{eq: adv}) at time level $t^n$. Assuming $\tilde u(x, t)$ is its entropy solution corresponding to the initial data $\tilde u(x, t^n)$, instantly we have
\begin{eqnarray}
\label{eq: eql1}
\bar {\tilde u}_j(t^n)=\bar u_j(t^n).
\end{eqnarray}
Since $f'(u) >0$, we have 
\begin{eqnarray}
\label{eq: eql2}
f(\tilde u(x_{j-\frac12}, t)) =f(\bar u_{j-1}(t^n)),
\end{eqnarray}
for $t\in [t^n, t^{n+1}]$. Since $\lambda \max_u |f'(u)|\le 1$, the characteristic starting from $x_{j-\frac12}$ would not hit the side $x_{j+\frac12}$, therefore 
\begin{eqnarray}
\label{eq: eql3}
\tilde u(x_{j+\frac12}, t)= u(x_{j+\frac12}, t)
\end{eqnarray}
for $t\in [t^n, t^{n+1}]$. Also since $\tilde u$ satisfies the maximum principle $\tilde u \le u_M$, we have
\begin{eqnarray*}
\bar {\tilde u}^{n+1}_j = \bar {\tilde u}^n_j -\lambda (\check {\tilde f}_{j+\frac12} - \check {\tilde f}_{j-\frac12}) \le u_M,
\end{eqnarray*}
where 
\begin{equation}
\label{cellaverage1}
\check {\tilde f}_{j-1/2} =\frac{1}{\Delta t} \int_{t^n}^{t^{n+1}} f(\tilde u(x_{j-1/2}, t)) dt. 
\end{equation}
Substituting (\ref{eq: eql1}), (\ref{eq: eql2}) and (\ref{eq: eql3}) into the above inequality, it follows that 
\begin{eqnarray*}
\bar {u}_j(t^n)-\lambda (\check {f}_{j+\frac12} - f(\bar u_{j-1}(t^n)) \le u_M.
\end{eqnarray*}
\end{proof}
For the case $f'(u)< 0$, we have the following 
\begin{thm}
\label{general_proof2}
Assuming $f'(u) < 0$ and $\lambda \max_u |f'(u)|\le 1$, we have 
\beq
\label{DL}
 \bar u_j (t^n)-\lambda (\check f_{j+\frac12} - f(\bar u_{j} (t^n))) \le u_M,
 \eeq 
if $u(x, t)$ is the entropy solution to problem (\ref{eq: adv}) subject to initial data $u_0(x)$.  
\end{thm}
\begin{proof}
The proof is similar. The only difference is that in this case, we shall consider an auxiliary problem (\ref{eq: adv}) with initial data 
\begin{eqnarray}
\tilde {\tilde u}(x, t^n) = \begin{cases} u(x, t^n) \quad  & x\ge x_{j+\frac12},\\
\bar u_{j} (t^n) \quad & x< x_{j+\frac12}.
\end{cases}
\end{eqnarray}
 \end{proof}
 Theorem \ref{general_proof} and \ref{general_proof2} implies the first {\bf main result} 
 \begin{thm}
 \label{UW}
 For the cases stated in Theorem \ref{general_proof} and \ref{general_proof2}: $f'(u)>0$ or $f'(u)<0$, with $\lambda \max_u |f'(u)|\le 1$, the estimate $$|\hat H^{rk}_{j+\frac12}-\tilde H^{rk}_{j+\frac12}|=\mathcal{O}(\Delta x^{2k+1} + \Delta t^p), \quad \forall j$$  holds if equation $$|\check f_{j+\frac12}-\hat H^{rk}_{j+\frac12}|=\mathcal{O}(\Delta x^{2k+1} + \Delta t^p), \quad \forall j$$ holds, when $\hat h_{j-\frac12}$ is the first order Godunov flux for the modification in (\ref{modified flux}). 
 \end{thm}
 \begin{proof}
 The theorem can be proved by combining earlier arguments in this section, observing that $\hat h_{j-\frac12} = f(\bar u^n_{j-1})$ if $f'(u)> 0$, otherwise $\hat h_{j-\frac12} = f(\bar u^n_{j})$.
 \end{proof}
 The conclusion from Theorem \ref{UW} is that the MPP flux limiters for high order FV RK scheme does not introduce extra CFL constraint to preserve the high order accuracy of the original scheme. In the linear advection case, Theorem \ref{UW} simply indicates that 
 \begin{rem}
 \label{DON}
The MPP flux limiters preserve high order accuracy under the CFL requirement $\lambda \max_u |f'(u)|\le 1$ for linear advection problems when high order numerical fluxes are limited to the first order upwind flux. Without much difficulty, we can generalize the results in Theorem \ref{general_proof}, \ref{general_proof2} to two dimensional linear advection problems. 
 \end{rem}
It is difficult to generalize the above approach to general convection-dominated diffusion problems. However, we believe this is one important step toward a complete proof. Below, by performing Taylor expansions around extrema, we provide a proof of \eqref{trc} with third order spatial and temporal accuracy ($k=1, p=3$) for a general nonlinear problem. We consider a first order monotone flux $\hat h_{j-\frac12}=\hat h(\bar u_{j-1}, \bar u_j)$ in the proposed parametrized flux limiting procedure \eqref{modified flux}. And we define 
\begin{equation}
L_{1,j} = \f{\hat h(\bar u_{j-1}, \bar u_j )-f(\bar u_{j-1})}{\bar u_{j}-\bar u_{j-1}}, \quad 
L_{2,j} = - \f{f(\bar u_{j})-\hat h(\bar u_{j-1}, \bar u_j )}{\bar u_{j}-\bar u_{j-1}},
\label{lipschitz}
\end{equation}
where $L_{1,j}$ and $L_{2,j}$ are two coefficients related to the monotonicity condition \cite{harten1983high}. Let $L=\max_j |L_{1,j}+L_{2,j}|$, we have
\begin{thm}
\label{thm: accuracy}
Consider a third order (or higher) finite volume RK discretization for a pure convection problem \eqref{eq: adv}, 
with a first order monotone flux $\hat h_{j-\frac12}=\hat h(\bar u_{j-1}, \bar u_j)$ in \eqref{modified flux}. The estimate \eqref{trc} holds with $k=1, p=3$ under the CFL condition $1-\lambda L \ge 0$.
\end{thm}

\begin{proof}
Using the earlier argument, we will only prove \eqref{3rd}, assuming $u_M-(\bar u_j-\lambda(\hat H^{rk}_{j+\frac{1}{2}}-\hat h_{j-\frac12})) < 0$. 
We mimic the proof for the finite difference scheme in \cite{mpp_xqx}.
First we use the 3-point Gauss Lobatto quadrature to approximate $\check f_{j+\f12}$, 
\begin{align}
\check f_{j+\f12}=\frac16 f(u(x_{j+\frac12},t^n+\Delta t))+\frac23 f((x_{j+\frac12},t^n+\frac{\Delta t}{2}))+\frac16 f((x_{j+\frac12},t^n))+\mathcal{O}(\Delta t^3). 
\label{glrule}
\end{align}
Following the characteristics, we get
\begin{align}
\check f_{j+\f12}=\frac16 f(u(x_{j+\frac{1}{2}}-\lambda_{1}\Delta x,t^n))+ \frac23 f(u(x_{j+\frac{1}{2}}-\lambda_{2}\Delta x,t^n )) + \frac16 f(u(x_{j+\frac{1}{2}},t^n))+\mathcal{O}(\Delta t^3),\label{glrule3}
\end{align}
where $\lambda_{1}$ and $\lambda_{2}$ can be determined from
\begin{eqnarray}
\lambda_{1}=\lambda f'(u(x_{j+\frac12}-\lambda_{1}\Delta x,t^n)), \quad
\lambda_{2}=\frac{\lambda}{2} f'(u(x_{j+\frac12}-\lambda_{2}\Delta x,t^n)).  \label{lamb}
\end{eqnarray}
For the finite volume method, $u(x^*,t^n)$ in (\ref{glrule3}) can be approximated by a second order polynomial reconstruction from $\bar u_{j-1}$, $\bar u_j$ and $\bar u_{j+1}$. Denoting $u_1=u(x_{j+\frac{1}{2}}-\lambda_{1}\Delta x,t^n)$, $u_2=u(x_{j+\frac{1}{2}}-\lambda_{2}\Delta x,t^n)$ and $u_3=u(x_{j+\frac{1}{2}},t^n)$, we have
\begin{subequations}
\label{3rdr}
\begin{align}
&u_1=\f16\left((5+6\lambda_1-6\lambda_1^2)\bar u_j+(-1+3\lambda_1^2)\bar u_{j-1}+(2-6\lambda_1+3\lambda_1^2)\bar u_{j+1}\right) +O(\Delta x^3), \\
&u_2=\f16\left((5+6\lambda_2-6\lambda_2^2)\bar u_j+(-1+3\lambda_2^2)\bar u_{j-1}+(2-6\lambda_2+3\lambda_1^2)\bar u_{j+1}\right) +O(\Delta x^3), \\
&u_3=\f16\left(5\bar u_j-\bar u_{j-1}+2\bar u_{j+1}\right) +O(\Delta x^3).
\end{align}
\end{subequations}
 

We prove (\ref{3rd}) case by case. We first consider the case $x_M \in I_j$, with $u_M=u(x_M)$, $u'_M=0$ and $u''_M\le 0$. We perform Taylor expansions around $x_M$,
\begin{subequations}
\label{tayloru}
\begin{align}
&\bar u_j=u_M+u'_M(x_j-x_M)+u''_M \left(\frac{(x_j-x_M)^2}{2}+\frac{\Delta x^2}{24}\right)+O(\Delta x^3), \label{tayloru0} \\
&\bar u_{j+1}=u_M+u'_M(x_j-x_M+\Delta x)+u''_M \left(\frac{(x_j-x_M+\Delta x)^2}{2}+\frac{\Delta x^2}{24}\right)+O(\Delta x^3), \label{taylorup} \\
&\bar u_{j-1}=u_M+u'_M(x_j-x_M)+u''_M \left(\frac{(x_j-x_M-\Delta x)^2}{2}+\frac{\Delta x^2}{24}\right)+O(\Delta x^3). \label{taylorum} 
\end{align}
\end{subequations}
Denoting $z=(x_j-x_M)/\Delta x$, the approximation in \eqref{3rdr} can be rewritten as
\begin{subequations}
\label{3rdr2}
\begin{align}
&u_1=u_M+u'_M\Delta x(\f12-\lambda_1+z)+u''_M\f{\Delta x^2}{2}(\f14-\lambda_1+\lambda_1^2+z-2\lambda_1z+z^2)+O(\Delta x^3), \\
&u_2=u_M+u'_M\Delta x(\f12-\lambda_2+z)+u''_M\f{\Delta x^2}{2}(\f14-\lambda_2+\lambda_2^2+z-2\lambda_2z+z^2)+O(\Delta x^3), \\
&u_3=u_M+u'_M\Delta x(\f12+z)+u''_M\f{\Delta x^2}{2}(\f14+z+z^2)+O(\Delta x^3).
\end{align}
\end{subequations}
Based on similar Taylor expansions of \eqref{tayloru}, for the flux function $f$,
from \eqref{tayloru} and \eqref{3rdr2}, we would have
\begin{subequations}
\label{taylorf}
\begin{align}
f(\bar u_j)&= f(u_M)+ f'(u_M) \Big(u'_M \Delta x z +u''_M\frac{\Delta x^2}{2}(\frac{1}{12}+z^2)\Big) \nonumber \\
&+\f12 f''(u_M)\Big(u'_M z \Delta x+u''_M\frac{\Delta x^2}{2}(\frac{1}{12}+z^2)\Big)^2+O(\Delta x^3), \label{taylorf0} \\
f(\bar u_{j-1})&= f(u_M)+ f'(u_M) \Big(u'_M \Delta x (z-1) +u''_M\frac{\Delta x^2}{2}(\frac{13}{12}-2z+z^2)\Big) \nonumber \\
&+\f12 f''(u_M)\Big(u'_M \Delta x (z-1) +u''_M\frac{\Delta x^2}{2}(\frac{13}{12}-2z+z^2)\Big)^2+O(\Delta x^3), \label{taylorfm} \\
f(u_1)&= f(u_M)+ f'(u_M) \Big(u'_M \Delta x (1/2-\lambda_1+z) +u''_M\frac{\Delta x^2}{2}(\frac{1}{4}-\lambda_1+\lambda_1^2+z-2 \lambda_1 z+z^2)\Big) \nonumber \\
&+\f12 f''(u_M)\Big(u'_M \Delta x (1/2-\lambda_1+z) +u''_M\frac{\Delta x^2}{2}(\frac{1}{4}-\lambda_1+\lambda_1^2+z-2 \lambda_1 z+z^2)\Big)^2+O(\Delta x^3), \label{taylorf1} \\
f(u_2)&= f(u_M)+ f'(u_M) \Big(u'_M \Delta x (1/2-\lambda_2+z) +u''_M\frac{\Delta x^2}{2}(\frac{1}{4}-\lambda_2+\lambda_2^2+z-2 \lambda_2 z+z^2)\Big) \nonumber \\
&+\f12 f''(u_M)\Big(u'_M \Delta x (1/2-\lambda_2+z) +u''_M\frac{\Delta x^2}{2}(\frac{1}{4}-\lambda_2+\lambda_2^2+z-2 \lambda_2 z+z^2)\Big)^2+O(\Delta x^3), \label{taylorf2} \\
f(u_3)&= f(u_M)+ f'(u_M) \Big(u'_M \Delta x (1/2+z) +u''_M\frac{\Delta x^2}{2}(\frac{1}{4}+z+z^2)\Big)\nonumber \\
&+\f12 f''(u_M)\Big(u'_M \Delta x (1/2+z) +u''_M\frac{\Delta x^2}{2}(\frac{1}{4}+z+z^2)\Big)^2+O(\Delta x^3). \label{taylorf3} 
\end{align}
\end{subequations}
Now denoting $\lambda_1 =\lambda_0+\eta_1 \Delta x+\mathcal{O}(\Delta x^2)$
and $\lambda_2=\frac{\lambda_0}{2}+\eta_2\Delta x+\mathcal{O}(\Delta x^2)$, where $\lambda_0=\lambda f'(u_M)$, we can determine $\eta_1$ and $\eta_2$ by substituting them into (\ref{lamb}) and we have
\begin{align*}
&\lambda_1=\lambda_0+f''(u_M)u'_M \lambda (z+\frac12-\lambda_0)\Delta x+\mathcal{O}(\Delta x^2), \nonumber \\
&\lambda_2=\frac{\lambda_0}{2}+f''(u_M)u'_M \frac{\lambda}{2} (z+\frac12-\frac{\lambda_0}{2})\Delta x+\mathcal{O}(\Delta x^2).
\end{align*}

For the first order monotone flux $\hat h_{j-\f12}=\hat h(\bar u_{j-1}, \bar u_j )$, it can be written as 
\beq
\hat h_{j-\f12}=f(\bar u_{j-1})+L_{1,j} (\bar u_{j}-\bar u_{j-1}), \quad L_{1,j} = \f{\hat h(\bar u_{j-1}, \bar u_j )-f(\bar u_{j-1})}{\bar u_{j}-\bar u_{j-1}},
\label{hupwind}
\eeq
where $f(\bar u_{j-1})=\hat h(\bar u_{j-1},\bar u_{j-1})$ due to consistence. $L_{1,j}$ is negative and bounded due to the monotonicity and Lipschitz continuous conditions. On the other hand, $\hat h_{j-\f12}$ 
can also be written as
\beq
\hat h_{j-\f12}=f(\bar u_{j})+L_{2,j} (\bar u_{j}-\bar u_{j-1}), \quad L_{2,j} = - \f{f(\bar u_{j})-\hat h(\bar u_{j-1}, \bar u_j )}{\bar u_{j}-\bar u_{j-1}},
\label{hdownwind}
\eeq
where $f(\bar u_{j})=\hat h(\bar u_{j},\bar u_{j})$, and $L_{2,j}$ is negative and bounded.

With above notations and $u'_M=0$, we now discuss the following two cases:


\bit
\item
If $f'(u_M)\ge 0$, we have $\lambda_0=\lambda f'(u_M)\in[0,1]$ since $\lambda \max_u |f'(u)| \le 1$. We take $\hat h_{j-\f12}$ as in \eqref{hupwind} and we have 
\beq
\bar u_j-\lambda\left(\check f_{j+\f12}-\hat h_{j-\frac12}\right)
=u_M+\frac{u''_M}{12}\Delta x^2 g(z,\lambda_0)+\mathcal{O}(\Delta x^3+\Delta t^3),
\label{star3n}
\eeq
where
\beq
g(z,\lambda_0)=g_1(z,\lambda_0)- 6\lambda L_{1,j} (1-2z),
\eeq
with
\beq
g_1(z,\lambda_0)=\f12+(5\lambda_0+3\lambda_0^2-2\lambda_0^3)+6(-3\lambda_0+\lambda_0^2)z+6z^2.
\eeq
$\lambda L_{1,j}(1-2z)\le 0$ for $z\in[-\f12, \f12]$ and $L_{1,j}\le 0$. The minimum value of function $g_1$ with respect to $z$ is
\beq
(g_1)_{min}=g_1(z,\lambda_0)\Big|_{z=-\frac12\lambda_0(\lambda_0-3)}=\f12+\frac{\lambda_0}{2}(\lambda_0-2)(\lambda_0-1)(5-3\lambda_0)\ge 0,
\eeq
so that $g(z,\lambda_0)\ge 0$. Since $u''_M\le 0$, from (\ref{star3n}) we obtain (\ref{3rd}).
\item
If $f'(u_M)<0$, we have $\lambda_0\in[-1,0]$. We take $\hat h_{j-\f12}$ in \eqref{hdownwind}, similarly we have (\ref{star3n}) and 
\beq
g(z,\lambda_0)=g_2(z,\lambda_0)-6\lambda L_{2,j} (1-2z),
\eeq
with
\beq
g_2(z,\lambda_0)=\f12+(-\lambda_0+3\lambda_0^2-2\lambda_0^3)+6(-\lambda_0+\lambda_0^2)z+6z^2.
\eeq
$\lambda L_{2,j}(1-2z)\le 0$ for $z\in[-\f12, \f12]$ and $L_{2,j}\le 0$. The minimum value of $g_2$ with respect to $z$ is
\beq
(g_2)_{min}=g_2(z,\lambda_0)\Big|_{z=-\frac12\lambda_0(\lambda_0-1)}
=\f12+\frac{\lambda_0}{2}(\lambda_0+1)(\lambda_0-1)(2-3\lambda_0)\ge 0,
\eeq
that is $g(z,\lambda_0) \ge 0$. Since $u''_M\le 0$, from (\ref{star3n}) we also obtain (\ref{3rd}).
\eit

Now if $x_M \notin I_j$, however there is a local maximum point $x^{loc}_M$ inside the cell of $I_j$, the above analysis still holds. We then consider that $u(x)$ reaches its local maximum $u^{loc}_M$ over $I_j$ at $x^{loc}_M=x_{j-\frac12}$, we have $u'_{j-\frac12}<0$. We take $\hat h_{j-\f12}$ as an average of \eqref{hupwind} and \eqref{hdownwind}. From the Taylor expansions in (\ref{taylorf}), following the same procedure as above, with $z=(x_j-x^{loc}_M)/\Delta x=(x_j-x_{j-\frac12})/\Delta x=1/2$, we have
\begin{align}
\bar u_j-\lambda\left(\check f_{j+\f12}-\hat h_{j-\frac12}\right) 
=u_{j-\frac12}+u'_{j-\frac12} \Delta x s_1 + (u'_{j-\frac12})^2 \Delta x^2 s_2 +u''_{j-\frac12}\frac{\Delta x^2}{2}s_3+\mathcal{O}(\Delta x^3+\Delta t^3),
\label{star3}
\end{align}
where
\begin{align*}
& s_1=\frac12(-2\lambda_0+\lambda_0^2)+\frac{1}{2}(1+\lambda(L_{1,j}+L_{2,j})), \\
& s_2=-f''(u_{j-\frac12})\frac{\lambda }{8}(3-4\lambda_0+4\lambda_0^2),\qquad 
  s_3=\frac{1}{3}(1-2\lambda_0+3\lambda_0^2-\lambda_0^3). 
\end{align*}
(\ref{star3}) can be rewritten as
\begin{align}
\bar u_j-\lambda\left(\check f_{j+\f12}-\hat h_{j-\frac12}\right)
=& u(x_{j-\frac12}-\sqrt{s_3}\Delta x)+u'_{j-\frac12}\Delta x \big(\frac12(-2\lambda_0+\lambda_0^2)+\sqrt{s_3}
\nonumber \\
+&\f12(1+\lambda(L_{1,j}+L_{2,j}))\big)+(u'_{j-\frac12})^2\Delta x^2 s_2 +\mathcal{O}(\Delta x^3+\Delta t^3).
\label{star31}
\end{align}
It is easy to check that $s_3 > 0$ and $\frac12(-2\lambda_0+\lambda_0^2)+\sqrt{s_3}>0$ for $\lambda_0=\lambda f'(u_M)\in[-1,1]$. From the CFL condition $1+\lambda (L_{1,j}+L_{2,j}) \ge 1-\lambda L \ge 0$, 
we obtain $u'_{j-\frac12}\Delta x \big(\frac12(-2\lambda_0+\lambda_0^2)+\sqrt{s_3}+\f12(1+\lambda(L_{1,j}+L_{2,j}))\big)\le 0$ since $u'_{j-\frac12}<0$.

Now to prove \eqref{3rd}, it is sufficient to show $u(x_{j-\frac12}-\sqrt{s_3}\Delta x)+\Delta x^2 (u'_{j-\frac12})^2 s_2\le u_M$ or $u'_{j-\frac12} = \mathcal{O}(\Delta x)$. If $[x_{j-\frac12}-\sqrt{s_3}\Delta x-\Delta x, x_{j-\frac12}-\sqrt{s_3}\Delta x]$ is not a monotone region, there is a point $x^{\#,1}$ in this region, such that $u'(x^{\#,1})=0$. Similarly,
if $[x_{j-\frac12}-\sqrt{s_3}\Delta x-\Delta x, x_{j-\frac12}-\sqrt{s_3}\Delta x]$ is a monotone increasing region, since $u'_{j-\frac12}<0$, there is one point $x^{\#,2}$ in $[x_{j-\frac12}-\sqrt{s_3}\Delta x, x_{j-\frac12}]$, such that $u'(x^{\#,2})=0$. For these two cases, $u'_{j-\frac12}=\mathcal{O}(\Delta x)$.
We then focus on the case when $[x_{j-\frac12}-\sqrt{s_3}\Delta x-\Delta x, x_{j-\frac12}-\sqrt{s_3}\Delta x]$ is a monotone decreasing region. We assume
\[
u(x_{j-\frac12}-\sqrt{s_3}\Delta x)+ c \Delta x^2 >u_M
\]
where $c=|(u'_{j-\frac12})^2 s_2|$. Since
\[
u(x_{j-\frac12}-\sqrt{s_3}\Delta x)=u(x_{j-\frac12}-\sqrt{s_3}\Delta x-\Delta x)+u'(x^{\#,3}) \Delta x,
\]
where $u'(x^{\#,3})<0$, we have
\[
u'(x^{\#,3}) \Delta x +c\Delta x^2>0,
\]
which implies $|u'(x^{\#,3})|\le c \Delta x$, therefore, $u'_{j-\frac12} =\mathcal{O}(\Delta x)$.

$x^{loc}_M=x_{j+\frac{1}{2}}$ with $u'_{j+\f12} \ge 0$ can be proved similarly. Combining the above discussion, (\ref{3rd}) is proved.
\end{proof}
Therefore, for the general nonlinear convection problem, the MPP flux limiters preserve the third order accuracy of the original FV RK scheme without extra CFL constraint. 

\begin{rem}
The above proof relies on characteristic tracing. It is difficult to directly generalize such approach to the convection-diffusion problem. On the other hand, similar strategy as that used in \cite{jiang2013parametrized} by using a Lax-Wendroff strategy, i.e. transforming temporal derivatives into spatial derivatives by repeating using PDEs and its differentiation versions, can be directly applied here. A similar conclusion can be obtained that {the MPP flux limiters preserve the third order accuracy of the original FV RK scheme for the convection dominated diffusion equation without extra CFL constraint}. To save some space, we will not repeat the algebraically tedious details here.
\end{rem}

\begin{rem}
It is technically difficult to generalize the proof in Theorem~\ref{thm: accuracy} to higher than third order, especially with the use of general monotone fluxes, for example, global Lax-Friedrich flux 
\beq
\label{eq: gLxf}
\hat h_{j-\frac12}=\hat h(\bar u_{j-1}, \bar u_j)=\frac12\big(f(\bar u_j)+f(\bar u_{j-1})-\alpha(\bar u_j-\bar u_{j-1})\big), \quad \alpha=\max_{u}|f'(u)|.
\eeq
On the other hand, the use of the global Lax-Friedrich flux with an extra large $\alpha$ is not unusual; yet it is quite involved to theoretically or numerically investigate such issue in a nonlinear system. Instead, we use a monotone but over-diffusive flux with 
\beq
\label{eq: over_diff}
\hat{h}_{j+\frac12} = \f12 \big((1+\alpha) \bar{u}_j + (1- \alpha) \bar{u}_{j+1}\big), \quad \alpha>\max_{u}|f'(u)|=1, 
\eeq
for a linear advection equation $u_t + u_x = 0$ with a set of carefully chosen initial conditions. Such scenario is set up to mimic the use of global Lax-Friedrich flux with an extra large $\alpha$ for general nonlinear systems. In Table~\ref{tab301}-\ref{tab303} below, we present the accuracy test for using the parametrized flux limiter with an over-diffusive first order monotone flux \eqref{eq: over_diff} with $\alpha=1.2$ on a linear 5th, 7th and 9th order FV RK schemes, which denoted to be ``FVRK5'', ``FVRK7'', ``FVRK9'' respectively. A mild CFL constraint around $0.7$ with time step $\Delta t=CFL \Delta x/\alpha$ is observed to be sufficient to maintain the high order accuracy of the underlying scheme with the MPP flux limiter. 
\end{rem}

\begin{table}[ht]\footnotesize
\centering
  \begin{tabular}{|c||c||c|c|c|c|c||c|c|}
    \hline
$CFL$& &   mesh &  $L^1$ error  & order  & $L^\infty$ error & order  & Umin    & Umax   \\ \hline
\multirow{14}{*}{$0.9$}
&\multirow{7}{*}{Non-}
  &    20 &     1.29E-02 &       --&     2.00E-02 &       -- &  -0.013805229 &   0.960012218  \\  \cline{3-9}
&\multirow{7}{*}{MPP}
  &    40 &     5.62E-04 &     4.52&     9.27E-04 &     4.43 &  -0.000670411 &   0.988524452  \\  \cline{3-9}
& &    80 &     1.87E-05 &     4.91&     3.13E-05 &     4.89 &  -0.000025527 &   0.998060523  \\  \cline{3-9}
& &   160 &     5.96E-07 &     4.97&     9.94E-07 &     4.98 &  -0.000000471 &   0.999076363  \\  \cline{3-9}
& &   320 &     1.87E-08 &     4.99&     3.12E-08 &     4.99 &  -0.000000025 &   0.999931894  \\  \cline{3-9}
& &   640 &     5.85E-10 &     5.00&     9.76E-10 &     5.00 &  -0.000000001 &   0.999980112  \\  \cline{3-9}
& &  1280 &     1.83E-11 &     5.00&     3.05E-11 &     5.00 &   0.000000000 &   0.999992161  \\  \cline{2-9}
&\multirow{7}{*}{MPP}
  &    20 &     9.97E-03 &       --&     1.82E-02 &       -- &   0.000000000 &   0.960132209  \\  \cline{3-9}
& &    40 &     5.52E-04 &     4.18&     1.31E-03 &     3.80 &   0.000000000 &   0.988525623  \\  \cline{3-9}
& &    80 &     1.89E-05 &     4.87&     4.62E-05 &     4.83 &   0.000000000 &   0.998060523  \\  \cline{3-9}
& &   160 &     6.04E-07 &     4.96&     2.01E-06 &     4.52 &   0.000000325 &   0.999076363  \\  \cline{3-9}
& &   320 &     1.91E-08 &     4.98&     7.25E-08 &     4.79 &   0.000000010 &   0.999931894  \\  \cline{3-9}
& &   640 &     6.04E-10 &     4.99&     2.95E-09 &     4.62 &   0.000000001 &   0.999980112  \\  \cline{3-9}
& &  1280 &     1.90E-11 &     4.99&     1.33E-10 &     4.47 &   0.000000000 &   0.999992161  \\  \hline
\multirow{14}{*}{$0.7$}
&\multirow{7}{*}{Non-}
  &    20 &     1.30E-02 &       --&     2.01E-02 &       -- &  -0.014015296 &   0.959761206  \\  \cline{3-9}
&\multirow{7}{*}{MPP} 
  &    40 &     5.66E-04 &     4.52&     9.35E-04 &     4.43 &  -0.000680048 &   0.988513480  \\  \cline{3-9}
& &    80 &     1.89E-05 &     4.90&     3.17E-05 &     4.88 &  -0.000025848 &   0.998060157  \\  \cline{3-9}
& &   160 &     6.03E-07 &     4.97&     1.01E-06 &     4.98 &  -0.000000482 &   0.999076351  \\  \cline{3-9}
& &   320 &     1.89E-08 &     4.99&     3.16E-08 &     4.99 &  -0.000000026 &   0.999931893  \\  \cline{3-9}
& &   640 &     5.92E-10 &     5.00&     9.87E-10 &     5.00 &  -0.000000001 &   0.999980112  \\  \cline{3-9}
& &  1280 &     1.85E-11 &     5.00&     3.09E-11 &     5.00 &   0.000000000 &   0.999992161  \\  \cline{2-9}
&\multirow{7}{*}{MPP}
  &    20 &     9.95E-03 &       --&     1.81E-02 &       -- &   0.000000000 &   0.959688278  \\  \cline{3-9}
& &    40 &     5.55E-04 &     4.16&     1.40E-03 &     3.70 &   0.000000000 &   0.988514505  \\  \cline{3-9}
& &    80 &     1.91E-05 &     4.86&     4.90E-05 &     4.84 &   0.000000000 &   0.998060157  \\  \cline{3-9}
& &   160 &     6.09E-07 &     4.97&     1.86E-06 &     4.72 &   0.000000000 &   0.999076351  \\  \cline{3-9}
& &   320 &     1.91E-08 &     5.00&     6.03E-08 &     4.94 &   0.000000002 &   0.999931893  \\  \cline{3-9}
& &   640 &     5.95E-10 &     5.00&     1.91E-09 &     4.98 &   0.000000000 &   0.999980112  \\  \cline{3-9}
& &  1280 &     1.85E-11 &     5.00&     5.61E-11 &     5.09 &   0.000000000 &   0.999992161  \\  \hline
  \end{tabular}
\caption{$L^1$ and $L^\infty$ errors and orders for $u_t+u_x=0$ with initial 
condition $u(x,0)=\sin^4(x)$. $T=1$. The over-diffusive global Lax-Friedrichs flux \eqref{eq: over_diff} is used 
with $\alpha=1.2$. FVRK5. }
\label{tab301}
\end{table}

\begin{table}[ht]\footnotesize
\centering
  \begin{tabular}{|c||c||c|c|c|c|c||c|c|}
    \hline
$CFL$& &    mesh &  $L^1$ error  & order  & $L^\infty$ error & order  & Umin    & Umax   \\ \hline
\multirow{14}{*}{$0.9$}
&\multirow{6}{*}{Non-}
  &   20 &     4.13E-03 &       --&     6.38E-03 &       -- &  -0.004489835 &   0.972363581  \\  \cline{3-9}
&\multirow{6}{*}{MPP} 
  &   40 &     4.69E-05 &     6.46&     7.37E-05 &     6.44 &  -0.000005603 &   0.989301523  \\  \cline{3-9}
& &   80 &     3.99E-07 &     6.88&     6.38E-07 &     6.85 &   0.000001412 &   0.998091183  \\  \cline{3-9}
& &  160 &     3.20E-09 &     6.96&     5.10E-09 &     6.97 &   0.000000392 &   0.999077344  \\  \cline{3-9}
& &  320 &     2.51E-11 &     6.99&     4.01E-11 &     6.99 &   0.000000002 &   0.999931925  \\  \cline{3-9}
& &  640 &     1.97E-13 &     7.00&     3.14E-13 &     6.99 &   0.000000000 &   0.999980113  \\  \cline{2-9}
&\multirow{6}{*}{MPP}
  &   20 &     3.60E-03 &       --&     6.39E-03 &       -- &   0.000517069 &   0.972406897  \\  \cline{3-9}
& &   40 &     4.78E-05 &     6.23&     1.04E-04 &     5.94 &   0.000064524 &   0.989302277  \\  \cline{3-9}
& &   80 &     6.29E-07 &     6.25&     2.95E-06 &     5.15 &   0.000003451 &   0.998091182  \\  \cline{3-9}
& &  160 &     1.42E-08 &     5.47&     2.09E-07 &     3.82 &   0.000000602 &   0.999077344  \\  \cline{3-9}
& &  320 &     4.87E-10 &     4.87&     1.44E-08 &     3.86 &   0.000000012 &   0.999931925  \\  \cline{3-9}
& &  640 &     1.78E-11 &     4.78&     1.01E-09 &     3.83 &   0.000000001 &   0.999980113  \\  \hline
\multirow{14}{*}{$0.7$}
&\multirow{6}{*}{Non-}
  &   20 &     4.12E-03 &       --&     6.38E-03 &       -- &  -0.004485289 &   0.972368315  \\  \cline{3-9}
&\multirow{6}{*}{MPP}
  &   40 &     4.69E-05 &     6.46&     7.37E-05 &     6.44 &  -0.000005556 &   0.989301572  \\  \cline{3-9}
& &   80 &     3.98E-07 &     6.88&     6.38E-07 &     6.85 &   0.000001412 &   0.998091183  \\  \cline{3-9}
& &  160 &     3.19E-09 &     6.96&     5.10E-09 &     6.97 &   0.000000392 &   0.999077344  \\  \cline{3-9}
& &  320 &     2.51E-11 &     6.99&     4.00E-11 &     6.99 &   0.000000002 &   0.999931925  \\  \cline{3-9}
& &  640 &     1.96E-13 &     7.00&     3.14E-13 &     7.00 &   0.000000000 &   0.999980113  \\  \cline{2-9}
&\multirow{6}{*}{MPP}
  &   20 &     3.62E-03 &       --&     6.59E-03 &       -- &   0.000515735 &   0.972263646  \\  \cline{3-9}
& &   40 &     4.65E-05 &     6.28&     8.94E-05 &     6.20 &   0.000054894 &   0.989301394  \\  \cline{3-9}
& &   80 &     3.98E-07 &     6.87&     6.38E-07 &     7.13 &   0.000001412 &   0.998091183  \\  \cline{3-9}
& &  160 &     3.19E-09 &     6.96&     5.10E-09 &     6.97 &   0.000000392 &   0.999077344  \\  \cline{3-9}
& &  320 &     2.51E-11 &     6.99&     4.00E-11 &     6.99 &   0.000000002 &   0.999931925  \\  \cline{3-9}
& &  640 &     1.96E-13 &     7.00&     3.14E-13 &     7.00 &   0.000000000 &   0.999980113  \\  \hline
  \end{tabular}
\caption{$L^1$ and $L^\infty$ errors and orders for $u_t+u_x=0$ with initial 
condition $u(x,0)=\sin^4(x)$. $T=1$. The over-diffusive global Lax-Friedrichs flux \eqref{eq: over_diff} is used 
with $\alpha=1.2$. FVRK7. }
\label{tab302}
\end{table}

\begin{table}[ht]\footnotesize
\centering
  \begin{tabular}{|c||c||c|c|c|c|c||c|c|}
    \hline
$CFL$& &    mesh &  $L^1$ error  & order  & $L^\infty$ error & order  & Umin    & Umax   \\ \hline
\multirow{14}{*}{$0.9$}
&\multirow{5}{*}{Non-}
  &   20 &     1.29E-03 &       --&     2.00E-03 &       -- &  -0.001216056 &   0.975890071  \\  \cline{3-9}
&\multirow{5}{*}{MPP}
  &   40 &     3.99E-06 &     8.34&     6.19E-06 &     8.34 &   0.000053321 &   0.989362841  \\  \cline{3-9}
& &   80 &     8.67E-09 &     8.85&     1.37E-08 &     8.82 &   0.000002016 &   0.998091807  \\  \cline{3-9}
& &  160 &     1.75E-11 &     8.95&     2.76E-11 &     8.96 &   0.000000397 &   0.999077349  \\  \cline{3-9}
& &  320 &     3.44E-14 &     8.99&     5.51E-14 &     8.97 &   0.000000002 &   0.999931925  \\  \cline{2-9}
&\multirow{5}{*}{MPP}
  &   20 &     1.20E-03 &       --&     2.37E-03 &       -- &   0.000393260 &   0.975868904  \\  \cline{3-9}
& &   40 &     8.91E-06 &     7.08&     3.54E-05 &     6.06 &   0.000092174 &   0.989363425  \\  \cline{3-9}
& &   80 &     2.90E-07 &     4.94&     2.72E-06 &     3.70 &   0.000003586 &   0.998091812  \\  \cline{3-9}
& &  160 &     1.15E-08 &     4.65&     2.02E-07 &     3.75 &   0.000000600 &   0.999077349  \\  \cline{3-9}
& &  320 &     4.32E-10 &     4.74&     1.30E-08 &     3.96 &   0.000000013 &   0.999931925  \\  \hline
\multirow{14}{*}{$0.7$}
&\multirow{5}{*}{Non-}
  &   20 &     1.29E-03 &       --&     2.00E-03 &       -- &  -0.001216106 &   0.975890020  \\  \cline{3-9}
&\multirow{5}{*}{MPP}
  &   40 &     3.99E-06 &     8.34&     6.19E-06 &     8.34 &   0.000053321 &   0.989362841  \\  \cline{3-9}
& &   80 &     8.67E-09 &     8.85&     1.37E-08 &     8.82 &   0.000002016 &   0.998091807  \\  \cline{3-9}
& &  160 &     1.75E-11 &     8.95&     2.76E-11 &     8.96 &   0.000000397 &   0.999077349  \\  \cline{3-9}
& &  320 &     3.44E-14 &     8.99&     5.60E-14 &     8.94 &   0.000000002 &   0.999931925  \\  \cline{2-9}
&\multirow{5}{*}{MPP}
  &   20 &     1.20E-03 &       --&     2.47E-03 &       -- &   0.000419926 &   0.975868183  \\  \cline{3-9}
& &   40 &     3.99E-06 &     8.23&     6.19E-06 &     8.64 &   0.000053321 &   0.989362841  \\  \cline{3-9}
& &   80 &     8.67E-09 &     8.85&     1.37E-08 &     8.82 &   0.000002016 &   0.998091807  \\  \cline{3-9}
& &  160 &     1.75E-11 &     8.95&     2.76E-11 &     8.96 &   0.000000397 &   0.999077349  \\  \cline{3-9}
& &  320 &     3.44E-14 &     8.99&     5.59E-14 &     8.95 &   0.000000002 &   0.999931925  \\  \hline
  \end{tabular}
\caption{$L^1$ and $L^\infty$ errors and orders for $u_t+u_x=0$ with initial 
condition $u(x,0)=\sin^4(x)$. $T=1$. The over-diffusive global Lax-Friedrichs flux \eqref{eq: over_diff} is used 
with $\alpha=1.2$. FVRK9. }
\label{tab303}
\end{table}
\section{Numerical simulations}
\label{sec4}
\setcounter{equation}{0}
\setcounter{figure}{0}
\setcounter{table}{0}

In this section, we present numerical tests of the proposed MPP high order FV RK WENO method for convection diffusion problems. Schemes with and without MPP limiters are compared.
In these tests, the time step size for the RK method is chosen such that
\begin{equation}
\Delta t = \min \Big( \frac{CFLC}{\max|f'(u)|}\Delta x , \frac{CFLD}{\max|a'(u)|} \Delta x^2 \Big),
\end{equation}
for one dimensional problems and
\begin{equation}
\Delta t = \min \Big( \frac{CFLC}{\max|f'(u)|/\Delta x + \max|g'(u)|/\Delta y} , \frac{CFLD}{\max|a'(u)|/\Delta x^2 + \max|b'(u)|/\Delta y^2} \Big),
\end{equation}
for two dimensional problems. Here CFLC (CFLD resp.) represents the CFL number for the convection (diffusion resp.) term.
In our tests, we take $CFLC = 0.6$ for convection-dominated problems and $CFLD = 0.8$ for pure diffusion problems. Herein we let ``MPP" and ``NonMPP" denote the scheme with and without the MPP limiter, and $U_{\max}\ (U_{\min} \ \mbox{resp.})$ denote the maximum (minimum resp.) value among the numerical cell averages $\bar{u}_j$. To better illustrate the effectiveness of the MPP limiters, we use linear weights instead of WENO weights in the reconstruction procedure for the convection term.

\subsection{Basic Tests}

\begin{exa}(1D Linear Problem)
\begin{equation}\label{1dlinear}
u_t+u_x=\epsilon u_{xx}, \ x\in [0, 2\pi], \ \epsilon = 0.00001.
\end{equation}
We test the proposed scheme on the problem (\ref{1dlinear}) with initial condition $u(x,0)=\sin^4(x)$ and periodic boundary condition. The exact solution is
\begin{equation}
\label{1daccuracy}
u(x,t)=\frac{3}{8} - \frac{1}{2}\exp(-4\epsilon t)\cos(2(x-t)) + \frac{1}{8}\exp(-16\epsilon t)\cos(4(x-t)).
\end{equation}
The $L_1$ and $L_\infty$ errors and orders of convergence for the scheme with and without MPP limiters are shown in Table \ref{tableexample1}. It is observed that the MPP limiter avoids overshooting and undershooting of the numerical solution while preserve high order accuracy.

\begin{table}[h]\footnotesize
\centering
        \begin{tabular}{|l  ||c   ||c  |c   ||c  |c   ||c  |c   |}

            \hline
                                                            & mesh                                  &$L_1$ error                                 &order                                  &$L_{\infty}$ error      &order      &Umax    &Umin         \\
            \hline
            \hline
            \multirow{6}{*}{Non-}    & \multicolumn{1}{l ||}{50}    & \multicolumn{1}{l |}{1.68E-04 }      & \multicolumn{1}{l ||}{---}        & \multicolumn{1}{l |}{2.76E-04 }      & \multicolumn{1}{l ||}{---} & \multicolumn{1}{l |}{0.996998594480 }      & \multicolumn{1}{l |}{-0.000182938402}\\
            \cline{2-8}
            \multirow{6}{*}{MPP}                                               & \multicolumn{1}{l ||}{100}    & \multicolumn{1}{l |}{5.47E-06 }      & \multicolumn{1}{l ||}{4.94 }      & \multicolumn{1}{l |}{9.11E-06 }      & \multicolumn{1}{l ||}{4.92  } & \multicolumn{1}{l |}{0.997933416789  }      & \multicolumn{1}{l |}{-0.000005718342  }\\
            \cline{2-8}
                                                           & \multicolumn{1}{l ||}{200}    & \multicolumn{1}{l |}{1.72E-07 }      & \multicolumn{1}{l ||}{4.99 }       & \multicolumn{1}{l |}{2.87E-07 }      & \multicolumn{1}{l ||}{4.99 } & \multicolumn{1}{l |}{0.999579130130 }      & \multicolumn{1}{l |}{-0.000000153518  }\\
            \cline{2-8}
                                                           & \multicolumn{1}{l ||}{400}    & \multicolumn{1}{l |}{5.38E-09 }      & \multicolumn{1}{l ||}{5.00 }       & \multicolumn{1}{l |}{9.00E-09 }      & \multicolumn{1}{l ||}{5.00} & \multicolumn{1}{l |}{0.999905929907 }      & \multicolumn{1}{l |}{-0.000000002134  }\\
            \cline{2-8}
                                                           & \multicolumn{1}{l ||}{800}    & \multicolumn{1}{l |}{1.68E-10 }      & \multicolumn{1}{l ||}{5.00}     & \multicolumn{1}{l |}{2.81E-10 }      & \multicolumn{1}{l ||}{5.00} & \multicolumn{1}{l |}{0.999945898951 }      & \multicolumn{1}{l |}{0.000000001890 }\\
            \cline{2-8}

            \hline
            \hline
            \multirow{6}{*}{MPP}    & \multicolumn{1}{l ||}{50}    & \multicolumn{1}{l |}{1.71E-04  }      & \multicolumn{1}{l ||}{---}     & \multicolumn{1}{l |}{2.87E-04 }      & \multicolumn{1}{l ||}{---} & \multicolumn{1}{l |}{0.996998296191 }      & \multicolumn{1}{l |}{0.000000000000  }\\
            \cline{2-8}
                                                           & \multicolumn{1}{l ||}{100}    & \multicolumn{1}{l |}{5.46E-06 }      & \multicolumn{1}{l ||}{4.93 }     & \multicolumn{1}{l |}{1.34E-05 }      & \multicolumn{1}{l ||}{4.42  } & \multicolumn{1}{l |}{ 0.997933416819 }      & \multicolumn{1}{l |}{ 0.000000016274  }\\
            \cline{2-8}
                                                           & \multicolumn{1}{l ||}{200}    & \multicolumn{1}{l |}{1.72E-07 }      & \multicolumn{1}{l ||}{ 5.00 }      & \multicolumn{1}{l |}{4.91E-07  }      & \multicolumn{1}{l ||}{4.77 } & \multicolumn{1}{l |}{0.999579130130 }      & \multicolumn{1}{l |}{0.000000013987 }\\
            \cline{2-8}
                                                           & \multicolumn{1}{l ||}{400}    & \multicolumn{1}{l |}{5.38E-09 }      & \multicolumn{1}{l ||}{ 5.03  }      & \multicolumn{1}{l |}{1.25E-08   }      & \multicolumn{1}{l ||}{5.29  } & \multicolumn{1}{l |}{0.999905929907 }      & \multicolumn{1}{l |}{0.000000001048   }\\
            \cline{2-8}
                                                           & \multicolumn{1}{l ||}{800}    & \multicolumn{1}{l |}{1.68E-10 }      & \multicolumn{1}{l ||}{5.01  }      & \multicolumn{1}{l |}{ 2.81E-10 }      & \multicolumn{1}{l ||}{5.48 } & \multicolumn{1}{l |}{ 0.999945898951  }      & \multicolumn{1}{l |}{ 0.000000001890  }\\
            \cline{2-8}

            \hline

        \end{tabular}

    \caption{Accuracy tests for 1D linear equation \eqref{1dlinear} with exact solution \eqref{1daccuracy} at time $T=1.0$.}
    \label{tableexample1}
\end{table}

We then test problem (\ref{1dlinear}) with the initial condition having rich solution structures
\begin{equation}
\label{1ddiscontinuous}
u_0(x)= \begin{cases}
                          \frac{1}{6}(G(x,\beta,z-\delta)+G(x,\beta,z+\delta)+4G(x,\beta,z)), \ \ \ &-0.8\le x \le -0.6; \\
                          1,                                 \ \ \                                 &-0.4\le x \le -0.2; \\
                          1-|10(x-0.1)|,                     \ \ \                                 &0 \le x \le 0.2;    \\
                          \frac{1}{6}(F(x,\gamma,a-\delta)+F(x,\gamma,a+\delta)+4F(x,\gamma,a)), \ \ \ &0.4\le x \le 0.6; \\
                          0, \ \ \                                                                &\mbox{ otherwise}.
                          \end{cases}
\end{equation}
where $G(x,\beta, z)=e^{-\beta(x-z)^2}$ and $F(x,\gamma,a)=\sqrt{\max(1-\gamma^2(x-a)^2,0)}$. The constants involved are $a=0.5, z=-0.7, \delta=0.005, \gamma=10$ and $\beta=\log 2/(36\delta^2)$ and the boundary condition is periodic. The maximum and minimum cell averages are listed in Table \ref{table1_1}. In Figure \ref{figure1}, the effectiveness of the MPP limiters in controlling the numerical solution within theoretical bounds can be clearly observed.
\begin{table}[h]\footnotesize
\centering

\begin{tabular}{|l ||c |c ||c |c |}
\hline
        & \multicolumn{2}{c ||}{NonMPP} & \multicolumn{2}{c|}{MPP} \\
\hline
  mesh  & Umax & Umin & Umax & Umin \\
\hline
   50 & 1.106238399422 &-0.114766938420   & 1.000000000000 & 0.000000000000  \\ \hline
   100 & 1.056114534445 &-0.067351423479   & 1.000000000000 & 0.000000000000  \\ \hline
   200 & 1.054864483784 &-0.054928012204   & 1.000000000000 & 0.000000000000  \\ \hline
   400 & 1.048250067722 &-0.048250171364   & 1.000000000000 & 0.000000000000  \\ \hline
   800 & 1.031246517796 &-0.031246517794   & 1.000000000000 & 0.000000000000  \\ \hline

\end{tabular}
\caption{The maximum and minimum values of the numerical cell averages for problem \eqref{1dlinear} with initial conditions \eqref{1ddiscontinuous} at time $T=1.0$.}
\label{table1_1}

\end{table}

\begin{figure}
\begin{center}
\includegraphics[width=3.0in]{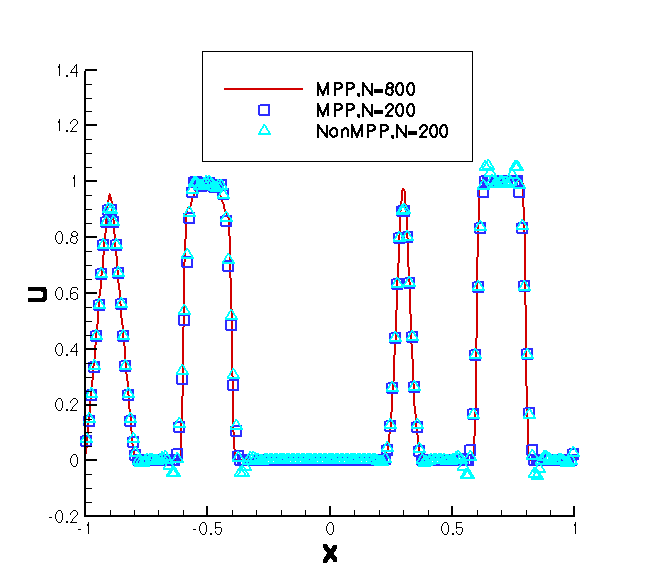}
\includegraphics[width=3.0in]{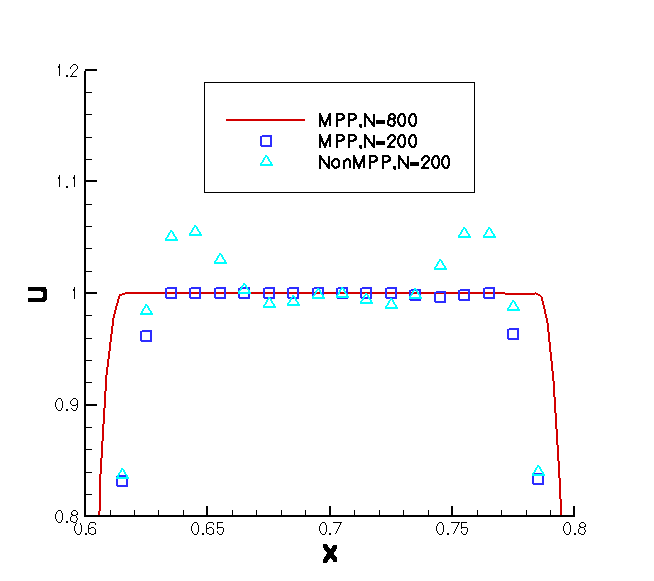}
\end{center}
\caption{Left: Comparison of the FV RK scheme with and without MPP limiters for 1d linear problem \eqref{1dlinear} with initial condition \eqref{1ddiscontinuous} at $T=1.0$. Right: Zoom-in around the overshooting.}
\label{figure1}
\end{figure}
\end{exa}

\begin{exa} (1D Nonlinear Equation)
We test the FV RK scheme with and without MPP limiters on Burgers' equation
\begin{equation}\label{burgersequation}
u_t+(\frac{u^2}{2})_x=\epsilon u_{xx},  \ x\in [-1,1], \ \epsilon = 0.0001,
\end{equation}
with initial condition
\begin{equation}\nonumber
u(x,0)=\begin{cases} 2,  \ \ \ &|x|<0.5; \\
              0, \ \ \ &\mbox{otherwise},
              \end{cases}
\end{equation}
and periodic boundary conditions. The results in Table \ref{tableexample2} shows that the numerical solution goes beyond the theoretical bounds if no limiters are applied and stays within the theoretical range if MPP limiters are applied.

\begin{table}[h]\footnotesize
\centering

\begin{tabular}{|l ||c |c ||c |c |}
\hline
        & \multicolumn{2}{c ||}{NonMPP} & \multicolumn{2}{c|}{MPP} \\
\hline
  mesh  & Umax & Umin & Umax & Umin \\
\hline
    50 & 2.349929038912 &-0.063536142936  & 1.818784698878 & 0.000000000000  \\ \hline
   100 & 2.438970633433 &-0.135799476071  & 1.879377697365 & 0.000000000000  \\ \hline
   200 & 2.217068598684 &-0.095548979222  & 1.913720603302 & 0.000000000000  \\ \hline
   400 & 2.216719764740 &-0.095114086983  & 1.938439146468 & 0.000000000000  \\ \hline
   800 & 2.210614277385 &-0.092745597929  & 1.959770865698 & 0.000000000000  \\ \hline
\end{tabular}
\caption{The maximum and minimum values of the numerical cell averages for Burgers' equation \eqref{burgersequation} at time $T=0.05$.}
\label{tableexample2}

\end{table}
\end{exa}

\begin{exa} (2D Linear Problem)
\begin{equation}\label{2dlinear}
u_t+u_x+u_y=\epsilon (u_{xx}+u_{yy}), \ \ \ (x,y)\in [0,2\pi]^2, \ \ \ \epsilon =0.001.
\end{equation}
We first consider the problem with initial condition $u(x,y,0)=\sin^4(x+y)$ and periodic boundary condition. The exact solution to the problem is
\begin{equation}
\label{2daccuracy}
u(x,y,t)=\frac{3}{8}-\frac{1}{2}\exp(-8\epsilon t)\cos(2(x+y-2t))+\frac{1}{8}\exp(-32\epsilon t)\cos(4(x+y-2t)).
\end{equation}
The $L_1$ and $L_\infty$ errors and orders of convergence for the FV RK scheme with and without MPP limiters are shown in Table \ref{tableexample3}. High order accuracy is preserved when the MPP limiters are applied to control the numerical solution within the theoretical bounds.
\begin{table}[h]\footnotesize
\centering
        \begin{tabular}{|l  ||c   ||c  |c   ||c  |c   ||c  |c   |}

            \hline
                                                            & mesh                                  &$L_1$ error                                 &order                                  &$L_{\infty}$ error      &order      &Umax    &Umin         \\
            \hline
            \hline

            \multirow{5}{*}{NonMPP}                                               & \multicolumn{1}{c ||}{$16\times 16$}    & \multicolumn{1}{l |}{4.86E-03 }      & \multicolumn{1}{l ||}{--- }      & \multicolumn{1}{l |}{9.30E-03 }      & \multicolumn{1}{l ||}{---  } & \multicolumn{1}{l |}{0.919696089900 }      & \multicolumn{1}{l |}{ 0.000159282060    }\\
            \cline{2-8}
                                                           & \multicolumn{1}{c ||}{$32\times 32$}    & \multicolumn{1}{l |}{2.85E-04 }      & \multicolumn{1}{l ||}{4.29 }       & \multicolumn{1}{l |}{4.49E-04 }      & \multicolumn{1}{l ||}{4.37  } & \multicolumn{1}{l |}{0.986054820018  }      & \multicolumn{1}{l |}{-0.000283832731  }\\
            \cline{2-8}
                                                           & \multicolumn{1}{c ||}{$64\times 64$}    & \multicolumn{1}{l |}{9.82E-06 }      & \multicolumn{1}{l ||}{4.84  }       & \multicolumn{1}{l |}{1.62E-05  }      & \multicolumn{1}{l ||}{4.79  } & \multicolumn{1}{l |}{0.995960434630  }      & \multicolumn{1}{l |}{-0.000004482350   }\\
            \cline{2-8}
                                                           & \multicolumn{1}{c ||}{$128\times 128$}    & \multicolumn{1}{l |}{3.12E-07  }      & \multicolumn{1}{l ||}{4.96 }     & \multicolumn{1}{l |}{5.22E-07 }      & \multicolumn{1}{l ||}{4.95  } & \multicolumn{1}{l |}{0.998407179488 }      & \multicolumn{1}{l |}{ 0.000001288422  }\\
            \cline{2-8}
                                                           & \multicolumn{1}{c ||}{$256\times 256$}    & \multicolumn{1}{l |}{9.73E-09 }      & \multicolumn{1}{l ||}{5.00 }      & \multicolumn{1}{l |}{1.63E-08 }      & \multicolumn{1}{l ||}{5.01  } & \multicolumn{1}{l |}{0.998990497491 }      & \multicolumn{1}{l |}{0.000000740680 }\\
            \hline
            \hline

            \multirow{5}{*}{MPP}     & \multicolumn{1}{c ||}{$16\times 16$}    & \multicolumn{1}{l |}{4.86E-03  }      & \multicolumn{1}{l ||}{ --- }     & \multicolumn{1}{l |}{ 9.30E-03 }      & \multicolumn{1}{l ||}{---  } & \multicolumn{1}{l |}{ 0.919696089900 }      & \multicolumn{1}{l |}{ 0.000159282060  }\\
            \cline{2-8}
                                                           & \multicolumn{1}{c ||}{$32\times 32$}    & \multicolumn{1}{l |}{2.87E-04 }      & \multicolumn{1}{l ||}{ 4.27 }      & \multicolumn{1}{l |}{4.49E-04  }      & \multicolumn{1}{l ||}{4.37  } & \multicolumn{1}{l |}{0.986054818813 }      & \multicolumn{1}{l |}{ 0.000000000000  }\\
            \cline{2-8}
                                                           & \multicolumn{1}{c ||}{$64\times 64$}    & \multicolumn{1}{l |}{9.82E-06 }      & \multicolumn{1}{l ||}{ 4.85  }      & \multicolumn{1}{l |}{1.64E-05   }      & \multicolumn{1}{l ||}{4.77    } & \multicolumn{1}{l |}{0.995960434630 }      & \multicolumn{1}{l |}{0.000000000000  }\\
            \cline{2-8}
                                                           & \multicolumn{1}{c ||}{$128\times 128$}    & \multicolumn{1}{l |}{3.12E-07 }      & \multicolumn{1}{l ||}{4.97 }      & \multicolumn{1}{l |}{ 5.22E-07 }      & \multicolumn{1}{l ||}{4.97   } & \multicolumn{1}{l |}{ 0.998407179488  }      & \multicolumn{1}{l |}{ 0.000001288422  }\\
            \cline{2-8}
                                                           & \multicolumn{1}{c ||}{$256\times 256$}    & \multicolumn{1}{l |}{9.73E-09 }      & \multicolumn{1}{l ||}{5.00 }     & \multicolumn{1}{l |}{1.63E-08 }      & \multicolumn{1}{l ||}{5.01  } & \multicolumn{1}{l |}{0.998990497491 }      & \multicolumn{1}{l |}{ 0.000000740680  }\\
            \hline

        \end{tabular}

    \caption{Accuracy tests for 2D linear equation \eqref{2dlinear} with exact solution \eqref{2daccuracy} at time $T=1.0$.}
    \label{tableexample3}
\end{table}

We then consider problem (\ref{2dlinear}) with initial condition
\begin{equation}
\label{2ddiscontinuous}
u(x,0)=\begin{cases} 1, \ \ \ \ &(x,y)\in [\frac{\pi}{2},\frac{3\pi}{2}]\times[\frac{\pi}{2},\frac{3\pi}{2}]; \\
              0, \ \ & \mbox{otherwise on }[0,2\pi]\times [0,2\pi],
              \end{cases}
\end{equation}
and periodic boundary condition. The results are shown in Table \ref{tableexample32}, which indicates the effectiveness of the MPP limiter.

\begin{table}[h]\footnotesize
\centering

\begin{tabular}{|c ||c |c ||c |c |}
\hline
        & \multicolumn{2}{c ||}{NonMPP} & \multicolumn{2}{c|}{MPP} \\
\hline
  mesh  & Umax & Umin & Umax & Umin \\
\hline
     $16\times16$  &  1.196476571354 &-0.102486638966    & 1.000000000000 & 0.000000000000  \\ \hline
    $32\times32$ & 1.317444117818 &-0.169214623680  & 1.000000000000 & 0.000000000000  \\  \hline
    $64\times64$ & 1.341696522446 &-0.182902057169    & 1.000000000000 & 0.000000000000  \\ \hline
    $128\times128$ &  1.225931525834 &-0.116989442889    & 1.000000000000 & 0.000000000000  \\ \hline
   $256\times256$ &  1.108731559448 &-0.055808238605  & 1.000000000000 & 0.000000000000  \\ \hline

\end{tabular}
\caption{Maximum and minimum cell averages in the 2D linear problem \eqref{2dlinear} with initial condition \eqref{2ddiscontinuous} at time $T=0.1$.}
\label{tableexample32}

\end{table}
\end{exa}

\begin{exa} (1D Buckley-Leverett Equation)
Consider the problem
\begin{equation}\label{1dbuckley-leverett}
u_t+f(u)_x=\epsilon(\nu(u)u_x)_x, \ \ \ \epsilon = 0.01,
\end{equation}
where
$$\nu(u)=\begin{cases} 4u(1-u), \ \ \ \ &0\le u \le 1; \\
              0, \ \ \ \ \ \ \ \ \ \ \ \ \ \ \ &\mbox{otherwise},
              \end{cases}
              \quad \mbox{and} \quad f(u)=\frac{u^2}{u^2+(1-u)^2}.
              $$
The initial condition is
$$u(x,0)=\begin{cases} 1-3x, \ \ \ \ 0\le x < \frac{1}{3}; \\
              0, \ \ \ \ \ \ \ \ \ \ \ \frac{1}{3} \le x \le 1,
              \end{cases}
              $$
and the boundary conditions are $u(0,t)=1$ and $u(1,t)=0$.
The numerical results are shown in Table \ref{tableexample4}. The numerical solution goes below $0$ if MPP limiters are not applied, and stays within the theoretical bounds $[0,1]$ when MPP limiters are applied. Figure \ref{figure4} illustrates the effectiveness of MPP limiters near the undershooting of the numerical solution.

\begin{table}[h]\footnotesize
\centering

\begin{tabular}{|l ||c |c ||c |c |}
\hline
        & \multicolumn{2}{c ||}{NonMPP} & \multicolumn{2}{c|}{MPP} \\
\hline
  mesh  & Umax & Umin & Umax & Umin \\
\hline

    50 &  1.000000000000000 &-0.002643266424381  & 1.000000000000000 & 0.000000000000000  \\ \hline
   100 &  1.000000000000000 &-0.001813338703220  & 1.000000000000000 & 0.000000000000000  \\ \hline
   200 &  1.000000000000000 &-0.000942402907667  & 1.000000000000000 & 0.000000000000000  \\ \hline
   400 &  1.000000000000000 &-0.000491323673758  & 1.000000000000000 & 0.000000000000000  \\ \hline
   800 &  1.000000000000000 &-0.000247268741213  & 1.000000000000000 & 0.000000000000000  \\ \hline

\end{tabular}
\caption{The maximum and minimum values for 1D Buckley-Leverett problem \eqref{1dbuckley-leverett} at time $T=0.2$.}
\label{tableexample4}

\end{table}

\begin{figure}
\begin{center}
\includegraphics[width=3.0in]{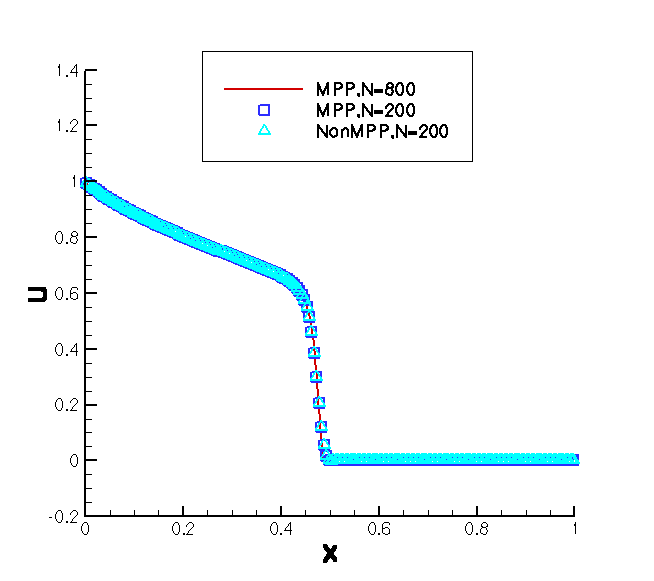}
\includegraphics[width=3.0in]{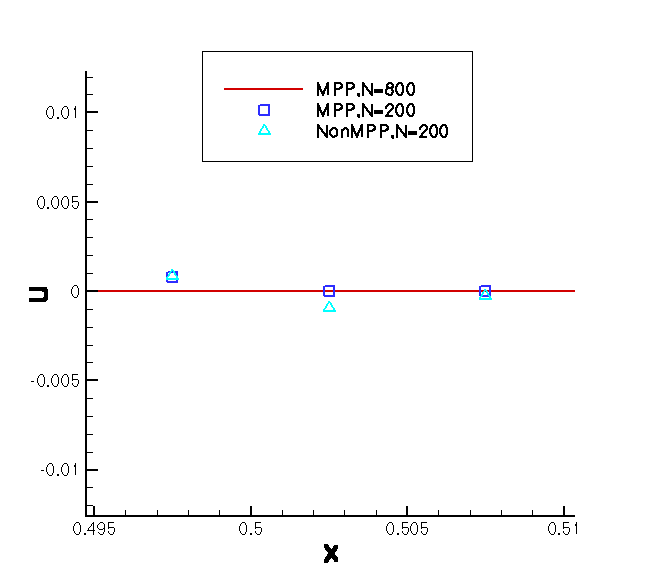}
\end{center}
\caption{Left: Solutions for 1D Buckley-Leverett equation \eqref{1dbuckley-leverett} at $T=0.2$. Right: Zoom-in around the undershooting.  }
\label{figure4}
\end{figure}

\end{exa}

\begin{exa}(2D Buckley-Leverett Equation)
Consider
\begin{equation}\label{2dbuckley-leverett}
u_t+f(u)_x+g(u)_y=\epsilon(u_{xx}+u_{yy}), \ \ \ (x,y)\in [-1.5,1.5]^2, \ \ \ \epsilon =0.01
\end{equation}
where
$$f(u)=\frac{u^2}{u^2+(1-u)^2},\ \ g(u)=f(u)(1-5(1-u)^2),$$
with initial condition
$$u(x,y,0)=\begin{cases} 1, \ \ \ \ &x^2+y^2<0.5; \\
              0, \ \ \ \ &\mbox{otherwise on }[-1.5,1.5]^2,
              \end{cases}
              $$
and periodic boundary conditions. The numerical results in Table \ref{tableexample5} show that the MPP limiters effectively control the numerical solution within the theoretical range $[0,1]$.

\begin{table}[h]\footnotesize
\centering

\begin{tabular}{|c ||c |c ||c |c |}
\hline
        & \multicolumn{2}{c ||}{NonMPP} & \multicolumn{2}{c|}{MPP} \\
\hline
  mesh  & Umax & Umin & Umax & Umin \\
\hline

  $16\times 16 $  & 1.190542402917 &-0.142603740886   & 1.000000000000 & 0.000000000000  \\ \hline
    $32\times 32$   & 1.183357844800 &-0.174592560044   & 1.000000000000 & 0.000000000000  \\ \hline
    $64\times 64$   & 1.148424330885 &-0.167227853261   & 1.000000000000 & 0.000000000000  \\ \hline
   $128\times 128$  & 1.084563025034 &-0.083883559766   & 1.000000000000 & 0.000000000000  \\ \hline
   $256\times 256$  & 0.998736899089 &-0.018463025969   & 0.998566263416 & 0.000000000000   \\ \hline

\end{tabular}
\caption{Maximum and minimum cell averages for 2D Buckley-Leverett problem \eqref{2dbuckley-leverett} at time $T=0.5$.}
\label{tableexample5}

\end{table}
\end{exa}

\begin{exa} (1D Porous Medium Equation)
Consider
\begin{equation}\label{1dporous}
u_t=(u^m)_{xx}, \ m>1, \ \ \ x\in [-2\pi,2\pi]
\end{equation}
whose solution is the Barenblatt solution in the following form
\begin{equation}
B_m(x,t)=t^{-k}\Big[(1-\frac{k(m-1)}{2m}\frac{|x|^2}{t^{2k}})_{+}\Big]^{\frac{1}{m+1}},
\end{equation}
with $k=\frac{1}{m+1}$ and $u_+=\max(u,0)$. The boundary conditions are assumed to be zero at both ends. Starting from time $T_0=1$, we compute the numerical solution of the problem up to time $T=2$ by the FV RK scheme and the results are shown in Table \ref{tableexample6}. Obviously, there are undershoots when regular FV RK scheme are applied. And the MPP limiters can effectively eliminate the overshoots in the numerical solution. Also the plot in Figure \ref{figure6} shows the effectiveness of the MPP limiters.

\begin{table}[h]\footnotesize
\centering

\begin{tabular}{|c ||c |c ||c |c |}
\hline
   $N=100$      & \multicolumn{2}{c ||}{NonMPP} & \multicolumn{2}{c|}{MPP} \\
\hline
  m   & Umax & Umin & Umax & Umin \\
\hline
2  & 0.793283780606 &-0.000338472445  & 0.793283375962 & 0.000000000000    \\ \hline

3  & 0.840666629482 &-0.001792679096  & 0.840663542409 & 0.000000000000    \\ \hline

5  & 0.890829374423 &-0.005693908465  & 0.890821177490 & 0.000000000000    \\ \hline

8  & 0.925837535365 &-0.003841778007  & 0.925826127818 & 0.000000000000    \\ \hline

\end{tabular}
\caption{Maximum and minimum cell average values for 1D porous medium problem \eqref{1dporous} with $m=2,3,5,8$ at time $T=2$.}
\label{tableexample6}

\end{table}

\begin{figure}
\begin{center}
\includegraphics[width=3.0in]{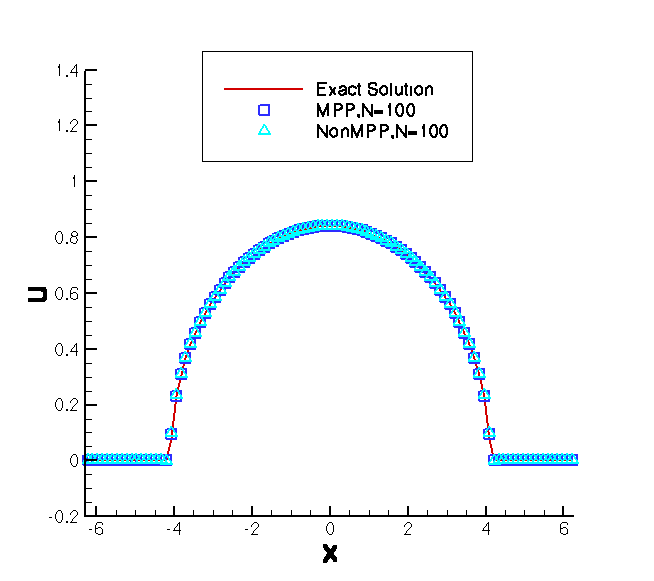}
\includegraphics[width=3.0in]{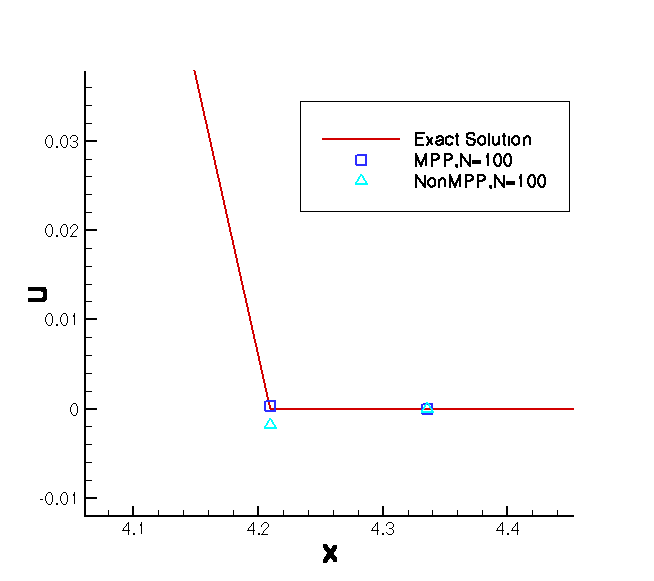}
\includegraphics[width=3.0in]{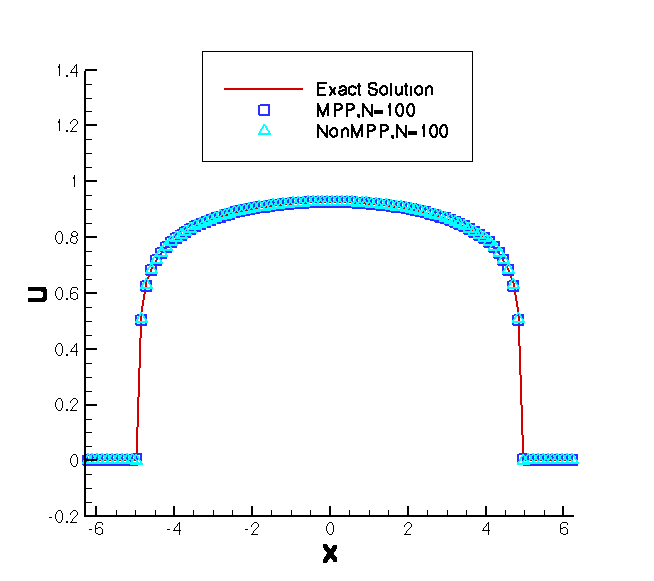}
\includegraphics[width=3.0in]{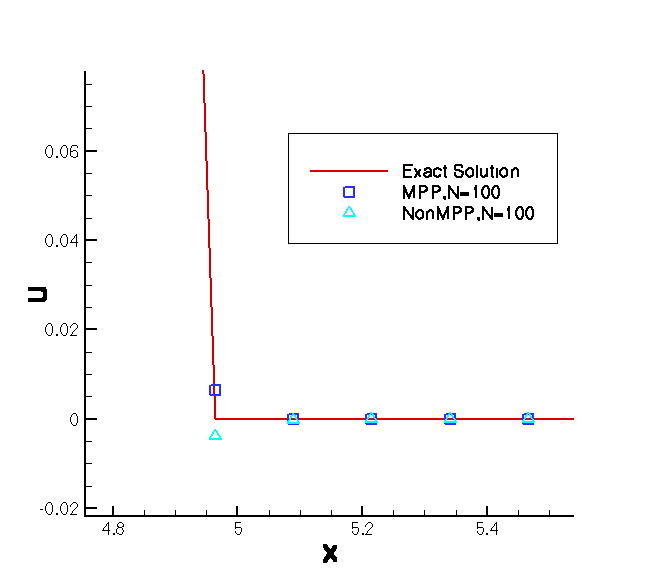}
\end{center}
\caption{Left: Plot for 1D porous medium problem \eqref{1dporous} with N=100 at $T=2$. Top is for m=3 and bottom is for m=8. Right: Zoom-in around the undershooting.}
\label{figure6}
\end{figure}
\end{exa}

\begin{exa}(2D Porous Medium Equation)
Consider
\begin{equation}\label{2dporous}
u_t=(u^m)_{xx}+(u^m)_{yy}, \ \ m=2, \ \ \ (x,y)\in [-1,1]^2
\end{equation}
with initial condition
$$u(x,y,0)=\begin{cases} 1, \ \ \ \ &(x,y)\in [-\frac{1}{2},\frac{1}{2}]^2; \\
              0, \ \ \ \ & \mbox{otherwise on } [-\frac{1}{2},\frac{1}{2}]^2,
              \end{cases}$$
and periodic boundary conditions. We produce the numerical results at time $T=0.005$, as shown in Table \ref{tableexample7}. The results show that the MPP limiters perform effectively at avoiding overshooting and undershooting of the numerical solution.

\begin{table}[h]\footnotesize
\centering

\begin{tabular}{|c ||c |c ||c |c |}
\hline
        & \multicolumn{2}{c ||}{NonMPP} & \multicolumn{2}{c|}{MPP} \\
\hline
  mesh  & Umax & Umin & Umax & Umin \\
\hline

    $16\times 16$  & 1.000485743751 &-0.000349298087  & 0.999827816078 & 0.000000000000  \\ \hline
    $32\times 32$  & 0.999625786453 &-0.001200636807  & 0.999573139639 & 0.000000000000  \\ \hline
    $64\times 64$  & 0.999537081790 &-0.000855830629   & 0.999533087178 & 0.000000000000  \\ \hline
   $128\times 128$ & 0.999527411822 &-0.000474775257   & 0.999526635569 & 0.000000000000  \\ \hline
   $256\times 256$ & 0.999525567240 &-0.000261471521   & 0.999525309113 & 0.000000000000  \\ \hline

\end{tabular}
\caption{Maximum and minimum cell average values for 2D porous medium problem \eqref{2dporous} at time $T=0.005$.}
\label{tableexample7}

\end{table}


\end{exa}

\subsection{Incompressible Flow Problems}
In this subsection, we test the proposed scheme on incompressible flow problems in the form
\begin{align}
&\omega_t+(u\omega)_x+(v\omega)_y=\frac{1}{Re}(\omega_{xx}+\omega_{yy}),
\end{align}
where $\langle u,v \rangle$ is the divergence-free velocity field and Re is the Reynold number. The theoretical solution satisfies the maximum principle due to the divergence-free property of the velocity field. For the numerical solution to satisfy the maximum principle, discretized divergence-free condition needs to be considered, hence special treatment needs to be taken when low order flux for the convection term is designed. For details, see \cite{mpp_xqx}, according to which we design the low order monotone flux for the following incompressible problems.

\noindent \begin{exa}(Rotation with Viscosity)
\begin{equation}\label{rbr}
u_t+(-yu)_x+(xu)_y=\frac{1}{Re}(u_{xx}+u_{yy}), \ \ (x,y)\in[-\pi,\pi]^2.
\end{equation}
The initial condition is shown in Figure \ref{figure8_1} and the boundary condition is assumed to be periodic. The numerical solution at time $T=0.1$ is shown in Table \ref{tableexample8}, which indicates that there are overshooting and undershooting in the numerical solution by regular FV RK scheme and they can be avoided by applying the MPP limiter. The solutions with and without MPP limiter are also compared in Figure \ref{figure8_2}. From Table \ref{tableexample8} and Figure \ref{figure8_2}, the effectiveness of the MPP limiter can be better illustrated when Renold number is larger. This is because the overshooting and undershooting are more apparent when Reynold number is larger, which corresponds to less diffusion.

\begin{figure}
\begin{center}
\includegraphics[width=3in]{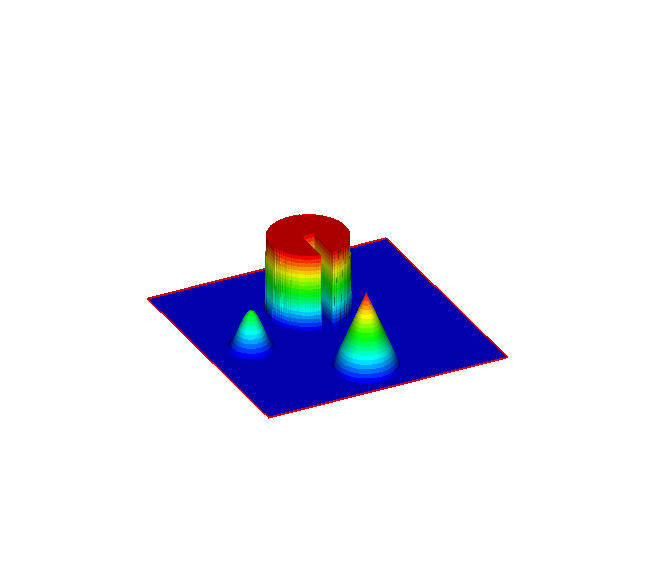}
\includegraphics[width=3in]{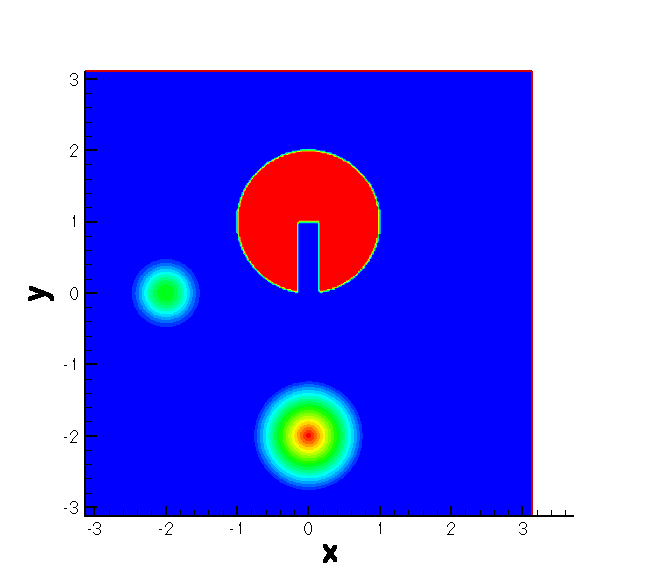}
\end{center}
\caption{Initial condition for Example 3.8 and Example 3.9.}
\label{figure8_1}
\end{figure}

\begin{table}[h]\footnotesize

\centering

\begin{tabular}{|c ||c |c ||c |c |}
\hline
Re=$100$    & \multicolumn{2}{c ||}{NonMPP} & \multicolumn{2}{c|}{MPP} \\
\hline
  mesh  & Umax & Umin & Umax & Umin \\
\hline
    $16\times  16 $ & 0.947915608973 &-0.041388485669    & 0.947719795318 & 0.000000000000    \\ \hline
    $32\times  32 $ & 0.999789765557 &-0.048836983632    & 0.996173203589 & 0.000000000000    \\ \hline
    $64\times 64  $ & 1.008171330748 &-0.039241271474    & 0.999999999928 & 0.000000000000    \\ \hline
   $128\times 128 $ & 1.002125190412 &-0.027962451582    & 0.999999999920 & 0.000000000000    \\ \hline
   $256\times 256 $ & 1.000099518450 &-0.012262487330    & 0.999999999983 & 0.000000000000    \\ \hline
\hline
Re=$10000$    & \multicolumn{2}{c ||}{NonMPP} & \multicolumn{2}{c|}{MPP} \\
\hline
  mesh  & Umax & Umin & Umax & Umin \\
\hline
    $16\times  16 $  & 0.949247968412 &-0.042285048496  & 0.949049295419 & 0.000000000000    \\ \hline
    $32\times  32 $  & 1.002247494119 &-0.053653247391  & 0.996943318800 & 0.000000000000    \\ \hline
    $64\times 64  $  & 1.012845607701 &-0.049914946698  & 0.999999462216 & 0.000000000000    \\ \hline
   $128\times 128 $  & 1.009050027036 &-0.050526262050  & 0.999999999977 & 0.000000000000    \\ \hline
   $256\times 256 $  & 1.007608558521 &-0.058482843302  & 0.999999999995 & 0.000000000000    \\ \hline
\end{tabular}
\caption{The maximum and minimum cell averages for rotation problem (\ref{rbr}) with two different Reynold numbers at $T=0.1$.}
\label{tableexample8}

\end{table}

\begin{figure}
\begin{center}
\includegraphics[width=3.0in]{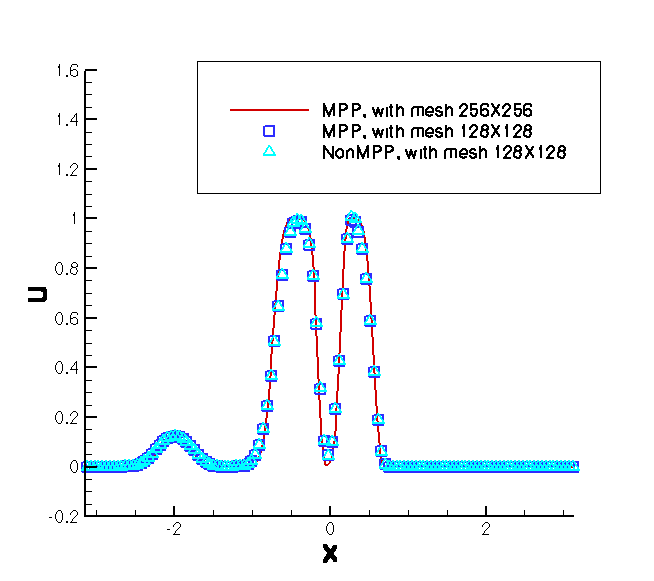}
\includegraphics[width=3.0in]{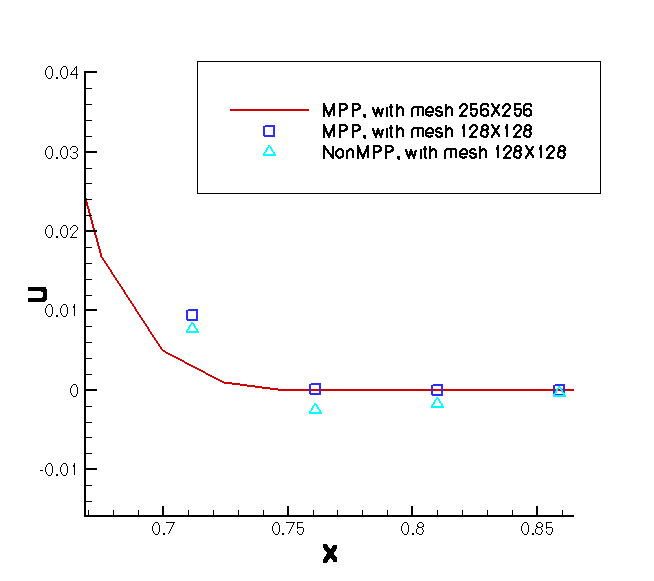}
\includegraphics[width=3.0in]{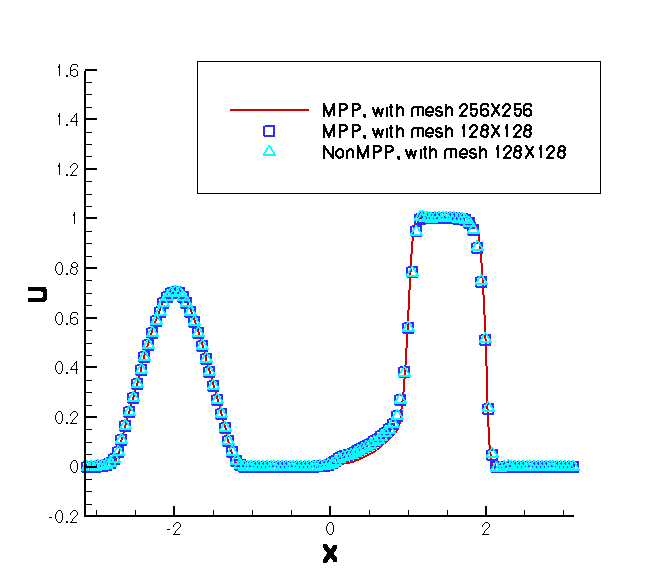}
\includegraphics[width=3.0in]{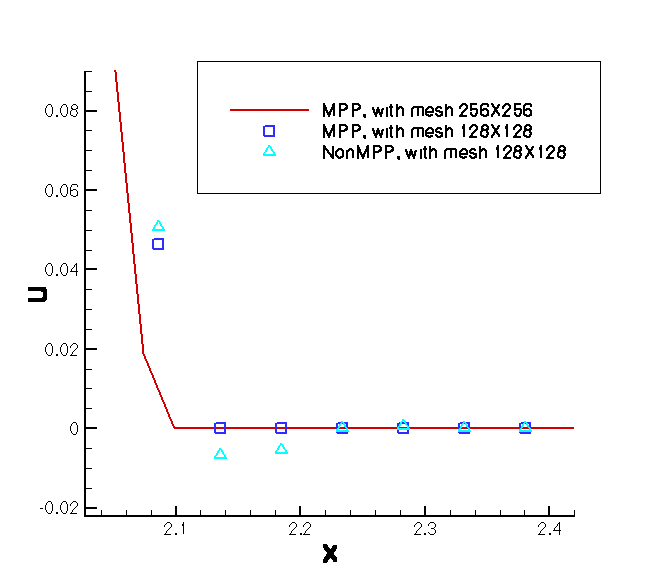}
\end{center}
\caption{Left: Cutting plots for rotation problem (\ref{rbr}) for Re=$10000$ at $T=0.1$. Right: Zoom-in around the undershooting. Top: cutting along $y=5\Delta y$ for $N_y=128$; Bottom: cutting along $x=0$. }
\label{figure8_2}
\end{figure}


\end{exa}

\begin{exa}(Swirling Deformation with Viscosity)
\begin{align}\label{sd}
&u_t+(-\cos^2(\frac{x}{2})\sin(y)g(t)u)_x+(\sin(x)\cos^2(\frac{y}{2})t(t))u)_y=\frac{1}{Re}(u_{xx}+u_{yy}),
\end{align}
where $(x,y)\in[-\pi,\pi]^2$ and $g(t)=\cos(\pi t /T)\pi$. The initial condition is the same as in Example 4.8 and the boundary conditions are also periodic. Similarly, we also compare the results for different Reynold numbers Re=$100$ and Re=$10000$. As shown in Table \ref{tableexample9}, the MPP limiter plays the role of eliminating overshooting and undershooting in the numerical solution, especially for problems with larger Reynold number. This can also be observed in Figure \ref{figure9_1}.

\begin{table}[h]\footnotesize
\centering

\begin{tabular}{|c ||c |c ||c |c |}
\hline
 Re=$100$       & \multicolumn{2}{c ||}{NonMPP} & \multicolumn{2}{c|}{MPP} \\
\hline
  mesh  & Umax & Umin & Umax & Umin \\
\hline
   $16\times  16 $ & 0.873440241699 &-0.010737472197  & 0.842184825192 & 0.000000000000  \\ \hline
    $32\times  32 $ & 0.971822334038 &-0.011947680561  & 0.942384582101 & 0.000000000000  \\ \hline
    $64\times 64  $ & 0.997563271155 &-0.005935366467  & 0.986960253479 & 0.000000000000  \\ \hline
   $128\times 128 $ & 1.000886437426 &-0.001258903421  & 0.998925498573 & 0.000000000000  \\ \hline
   $256\times 256 $ & 1.000040508119 &-0.000036182185  & 0.999992956155 & 0.000000000000  \\ \hline
\hline
Re=$10000$       & \multicolumn{2}{c ||}{NonMPP} & \multicolumn{2}{c|}{MPP} \\
\hline
  mesh  & Umax & Umin & Umax & Umin \\
\hline
    $16\times  16 $ & 0.874953790056 &-0.011212471543  & 0.846813512747 & 0.000000000000  \\ \hline
    $32\times  32 $ & 0.973964125865 &-0.014299538733  & 0.942368749644 & 0.000000000000    \\ \hline
    $64\times 64  $ & 1.000873875979 &-0.006640227946  & 0.988604733672 & 0.000000000000    \\ \hline
   $128\times 128 $ & 1.002350640870 &-0.002755842119  & 0.999375840770 & 0.000000000000  \\ \hline
   $256\times 256 $ & 1.000734372263 &-0.000563730690  & 0.999998986667 & 0.000000000000    \\ \hline
\end{tabular}
\caption{The maximum and minimum cell averages for swirling deformation problem (\ref{sd}) with two different Reynold numbers at T=$0.1$.}
\label{tableexample9}

\end{table}

\begin{figure}
\begin{center}
\includegraphics[width=3.0in]{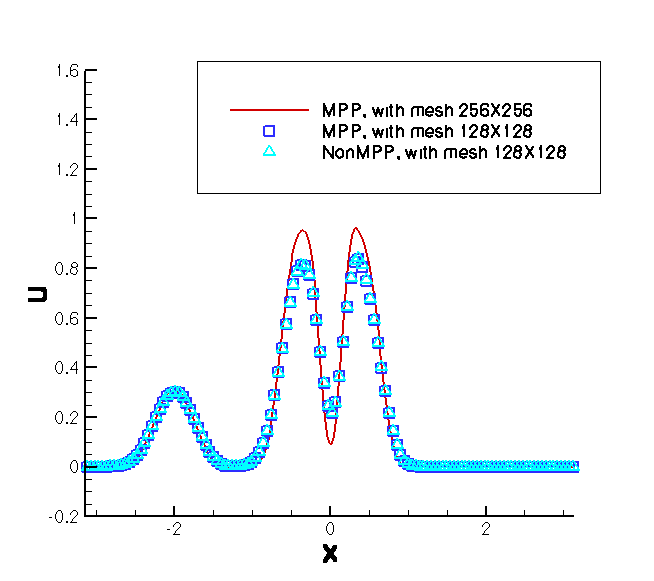}
\includegraphics[width=3.0in]{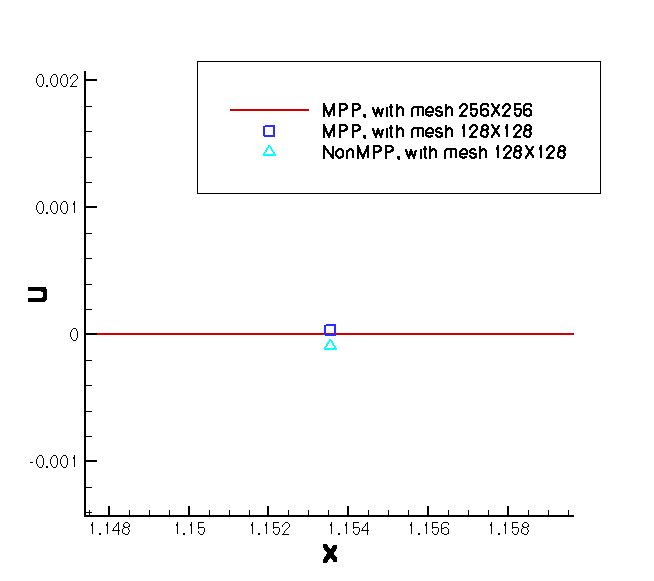}
\includegraphics[width=3.0in]{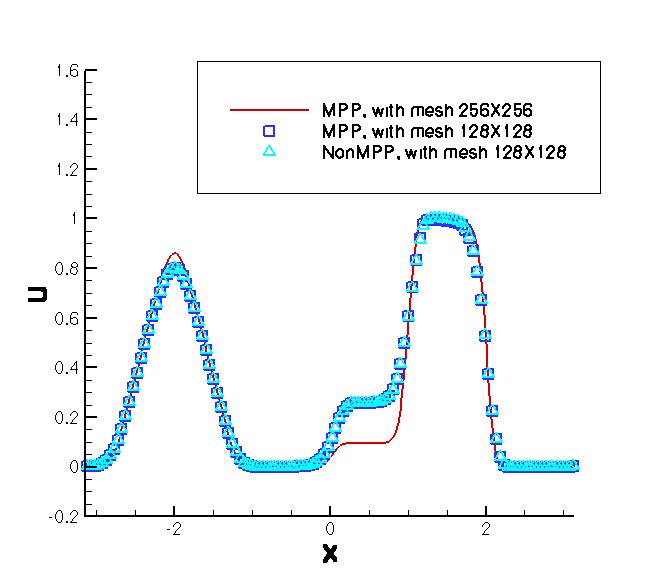}
\includegraphics[width=3.0in]{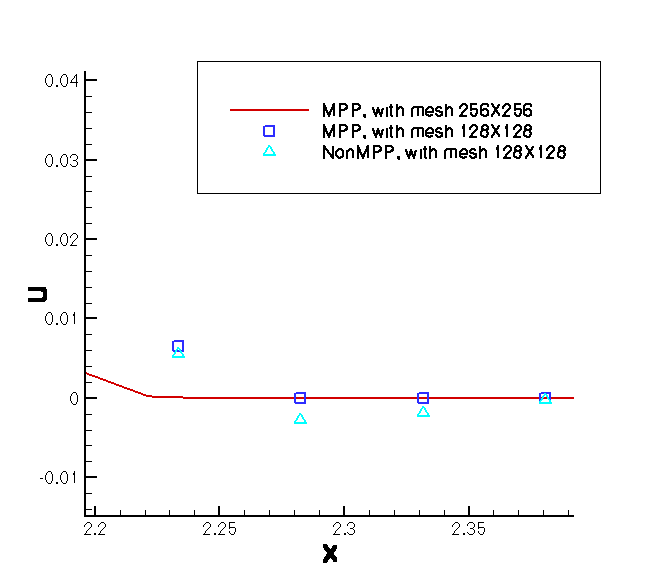}
\end{center}
\caption{Left: Cutting plots for swirling deformation problem (\ref{sd}) for Re=$10000$ at T=$0.1$. Right: Zoom-in around the undershooting. Top: cutting along $y=5\Delta y$ for $N_y=128$; Bottom: cutting along $x=0$. }
\label{figure9_1}
\end{figure}


\end{exa}






\begin{exa}(Vortex Patch)
Consider the problem
\begin{align}\label{vp}
&\omega_t+(u\omega)_x+(v\omega)_y=\frac{1}{Re}(\omega_{xx}+\omega_{yy}), \\
& \Delta \psi = \omega,\  \langle u,v \rangle = \langle -\psi_y, \psi_x \rangle,
\end{align}
with the following initial condition
\begin{equation}
\omega(x,y,0)=\begin{cases}
                       -1, \ &\frac{\pi}{2} \le x \le \frac{3\pi}{2}, \ \frac{\pi}{4} \le \frac{3\pi}{4}, \\
                       1, \ \ &\frac{\pi}{2} \le x \le \frac{3\pi}{2}, \ \frac{5\pi}{4} \le \frac{7\pi}{4}, \\
                       0, \ \ &\mbox{otherwise},
                       \end{cases}
\end{equation}
and periodic boundary condition. The maximum and minimum cell averages of the numerical solution with two Reynold numbers Re=$100$ and Re=$10000$, obtained by regular FV RK scheme and the scheme with the MPP limiter are compared in Table \ref{tableexample10}, from which we can observe the effectiveness of the MPP limiter in controlling overshooting and undershooting in the numerical solution. The contour plot of the solution is presented in Figure \ref{figure10}, which shows that the solution obtained by FV RK scheme with the MPP limiter is comparable to that obtained by regular FV RK scheme.

\begin{table}[h]\footnotesize
\centering

\begin{tabular}{|c ||c |c ||c |c |}
\hline
 Re=$100$       & \multicolumn{2}{c ||}{NonMPP} & \multicolumn{2}{c|}{MPP} \\
\hline
  mesh  & Umax & Umin & Umax & Umin \\
\hline

    $16\times  16 $ & 1.035853749815 &-1.035699868274   & 1.000000000000 &-1.000000000000    \\ \hline
    $32\times  32 $ & 1.054573231517 &-1.054663726026   & 1.000000000000 &-1.000000000000    \\ \hline
    $64\times 64  $ & 1.044017351861 &-1.044000125346   & 1.000000000000 &-1.000000000000    \\ \hline
   $128\times 128 $ & 1.010637311054 &-1.010641150928   & 1.000000000000 &-1.000000000000     \\ \hline
   $256\times 256 $ & 1.000000232315 &-1.000000231632   & 1.000000000000 &-1.000000000000     \\ \hline
\hline
Re=$10000$       & \multicolumn{2}{c ||}{NonMPP} & \multicolumn{2}{c|}{MPP} \\
\hline
  mesh  & Umax & Umin & Umax & Umin \\
\hline

    $16\times  16 $ & 1.036117022938 &-1.035951331163  & 1.000000000000 &-1.000000000000  \\ \hline
    $32\times  32 $ & 1.060652217270 &-1.060764279809  & 1.000000000000 &-1.000000000000  \\ \hline
    $64\times 64  $ & 1.086490500643 &-1.086296444198  & 1.000000000000 &-1.000000000000   \\ \hline
   $128\times 128 $ & 1.127323843780 &-1.127407543973  & 1.000000000000 &-1.000000000000  \\ \hline
   $256\times 256 $ & 1.129384376147 &-1.129395445889  & 1.000000000000 &-1.000000000000   \\ \hline

\end{tabular}

\caption{The maximum and minimum cell averages for vortex patch problem (\ref{vp}) at time T=$0.1$ with Re=$100$ and Re=$10000$. }
\label{tableexample10}

\end{table}


\begin{figure}
\centering
\includegraphics[width=3in]{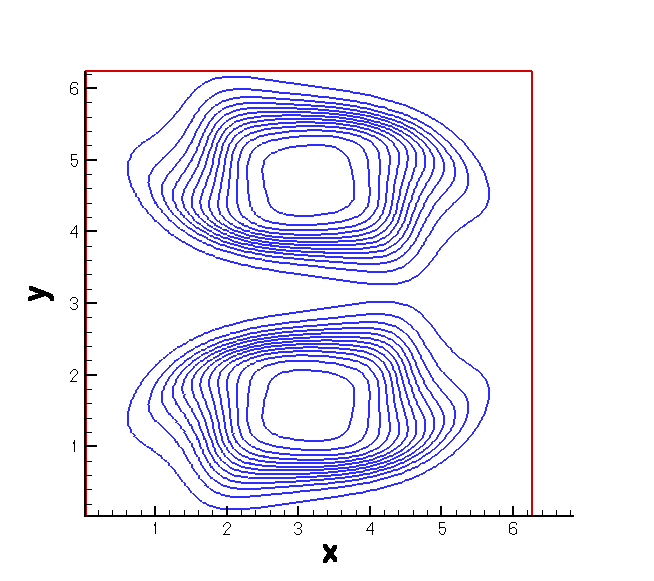},
\includegraphics[width=3in]{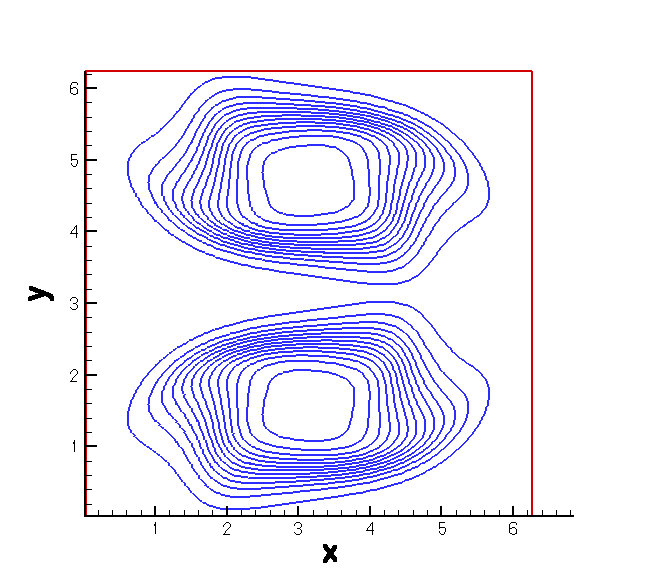}\\
\includegraphics[width=3in]{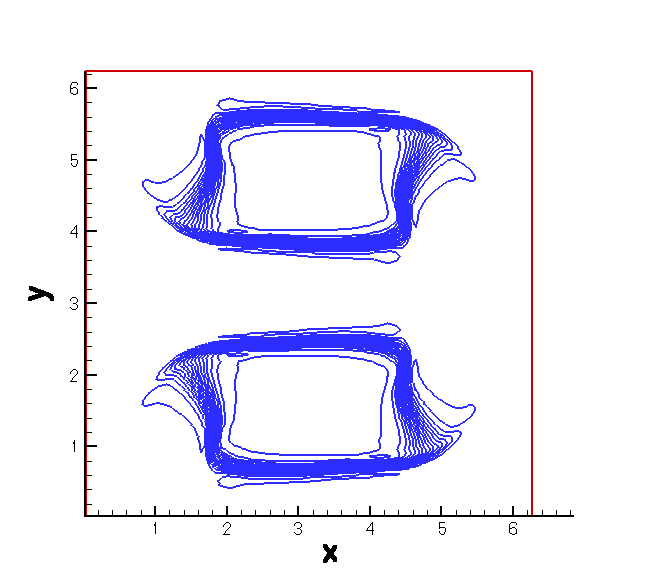},
\includegraphics[width=3in]{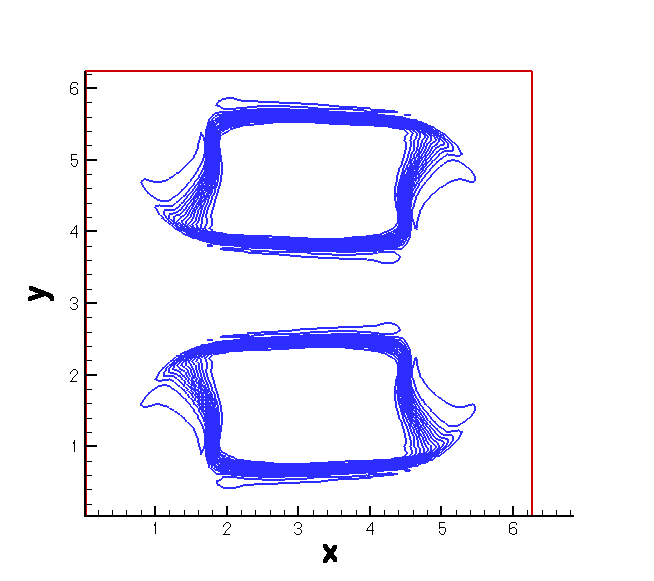}
\caption{Contours of the numerical solution for vortex patch problem (\ref{vp}) with Re=$100$ (top) and Re=$10000$ (bottom) at time $T=5$. The contours on the left are for the NonMPP scheme and those on the right are for the MPP scheme. 30 equally spaced contour lines within the range $[-1.1,1.1]$ are plotted.}
\label{figure10}
\end{figure}

\end{exa}

\section{Conclusion}
\label{sec6}
\setcounter{equation}{0}
\setcounter{figure}{0}
\setcounter{table}{0}

In this paper, we have successfully generalized the MPP flux limiters to the high order FV RK WENO schemes solving convection-dominated problems. For a special case, $f'(u)>0$ or $f'(u)<0$, we provide a complete analysis that the original high order FV RK WENO scheme coupled with the MPP flux limiters maintains high order accuracy and MPP property when Godunov type flux is used as the first order flux, toward which the high order numerical flux is limited. For a general setting, we rely on the Taylor expansion around extrema to prove that the FV RK schemes with MPP flux limiters preserve up to third order accuracy without addition CFL constraint. Establishing analysis for accuracy preservation under suitable constraints in a general setting will be part of our future work.



\bibliographystyle{siam}
\bibliography{refer}

\end{document}